\documentclass[11pt]{article}
\usepackage[top=1in, bottom=1in, left=1in, right=1in]{geometry}

\usepackage[linktocpage,colorlinks,linkcolor=blue,anchorcolor=blue,citecolor=blue,urlcolor=blue,pagebackref]{hyperref}
\usepackage{euscript,mdframed}
\usepackage{microtype,todonotes,relsize}
\usepackage{amsmath,amsthm}
\usepackage{algorithmic}

\usepackage{accents}
\usepackage{comment,url,graphicx,relsize}
\usepackage{amssymb,amsfonts,amsmath,amsthm,amscd,dsfont,mathrsfs,mathtools,nicefrac, bm}
\usepackage{float,psfrag,epsfig,color,xcolor,url}
\usepackage{epstopdf,bbm,mathtools,enumitem}
\usepackage{subfigure}
\usepackage[ruled,vlined]{algorithm2e}
\usepackage{tablefootnote}
\graphicspath{{Plots/}}
\usepackage{diagbox}
\usepackage{tikz}
\usetikzlibrary{calc,patterns,angles,quotes}
\usepackage{makecell}
\usepackage{multirow}

\newcommand{\grad}{\mathsf{grad}} 

\newcommand{\M}{\mathcal{M}}

\newcommand{\E}{\mathbb{E}}
\newcommand{\RR}{\mathbb{R}}

\providecommand{\diag}{\mathop\mathrm{diag}}
\providecommand{\tr}{\mathop\mathrm{tr}}

\newcommand{\bxi}{\bar{\xi}}
\newcommand{\bet}{\bar{\eta}}
\newcommand{\retr}{\mathsf{Retr}}

\newcommand{\Exp}{\mathsf{Exp}}
\newcommand{\proj}{\mathsf{proj}}

\newcommand{\St}{\operatorname{St}}
\newcommand{\Gr}{\operatorname{Gr}}
\newcommand{\T}{\operatorname{T}}

\ifdefined\nonewproofenvironments\else
\ifdefined\ispres\else
\renewenvironment{proof}{\noindent\textbf{Proof.}\hspace*{.3em}}{\qed\\}
\newenvironment{proof-sketch}{\noindent\textbf{Proof Sketch}
  \hspace*{0.em}}{\qed\bigskip\\}
\newenvironment{proof-idea}{\noindent\textbf{Proof Idea}
  \hspace*{0.em}}{\qed\bigskip\\}
\newenvironment{proof-of-lemma}[1][{}]{\noindent\textbf{Proof of Lemma {#1}.}
  \hspace*{0.em}}{\qed\\}
\newenvironment{proof-of-corollary}[1][{}]{\noindent\textbf{Proof of Corollary {#1}.}
  \hspace*{0.em}}{\qed\\}
\newenvironment{proof-of-theorem}[1][{}]{\noindent\textbf{Proof of Theorem {#1}.}
  \hspace*{0.em}}{\qed\\}
\newenvironment{proof-attempt}{\noindent\textbf{Proof Attempt}
  \hspace*{0.em}}{\qed\bigskip\\}

\fi
\newtheorem{theorem}{Theorem}[section]
\newtheorem{lemma}{Lemma}[section]

\newtheorem{proposition}{Proposition}[section]
\newtheorem{assumption}{Assumption}[section]
\newtheorem{remark}{Remark}[section]

\newtheorem{definition}{Definition}[section]

\usetikzlibrary{calc}

\allowdisplaybreaks
\usepackage[authoryear,round]{natbib}
\renewcommand*{\backref}[1]{\ifx#1\relax \else Page #1 \fi}
\renewcommand*{\backrefalt}[4]{%
  \ifcase #1 \footnotesize{(Not cited.)}%
  \or        \footnotesize{(Cited on page~#2.)}%
  \else      \footnotesize{(Cited on pages~#2.)}%
  \fi
}

\newcommand*{\colorboxed}{}
\def\colorboxed#1#{%
  \colorboxedAux{#1}%
}
\newcommand*{\colorboxedAux}[3]{%
  \begingroup
    \colorlet{cb@saved}{.}%
    \color#1{#2}%
    \boxed{%
      \color{cb@saved}%
      #3%
    }%
  \endgroup
}
\numberwithin{equation}{section}
\usepackage{xcolor} 
\usepackage{marginnote}

\newcommand{\todol}[2][]{{%
 \let\marginpar\marginnote
 \reversemarginpar
 \renewcommand{\baselinestretch}{0.8}%
 \todo[color=yellow]{#2}}}

\title{Zeroth-order Riemannian Averaging Stochastic Approximation Algorithms}
\author{
{Jiaxiang Li} \thanks{Department of Mathematics, University of California, Davis.  \texttt{jxjli@ucdavis.edu}}
\and
Krishnakumar Balasubramanian\thanks{Department of Statistics, University of California, Davis.  \texttt{kbala@ucdavis.edu}}
\and
Shiqian Ma \thanks{Department of Computational Applied Math and Operations Research, Rice University.  \texttt{sqma@rice.edu}}
\and
}
\date{}

\begin{document}
\maketitle

\begin{abstract}
We present Zeroth-order Riemannian Averaging Stochastic Approximation (\texttt{Zo-RASA}) algorithms for stochastic optimization on Riemannian manifolds. We show that \texttt{Zo-RASA} achieves optimal sample complexities for generating  $\epsilon$-approximation first-order stationary solutions using only one-sample or constant-order batches in each iteration. Our approach employs Riemannian moving-average stochastic gradient estimators, and a novel Riemannian-Lyapunov analysis technique for convergence analysis. We improve the algorithm's practicality by using retractions and vector transport, instead of exponential mappings and parallel transports, thereby reducing per-iteration complexity. Additionally, we introduce a novel geometric condition, satisfied by manifolds with bounded second fundamental form, which enables new error bounds for approximating parallel transport with vector transport. 
\end{abstract}

\section{Introduction} 
We consider zeroth-order algorithms for solving the following Riemannian optimization problem,
\begin{align}\label{stochastic_problem}
    \min_{x\in\M} f(x):=\E_{\xi}[F(x,\xi)],
\end{align}
where $\M$ is a $d$-dimensional complete manifold, $f:\M\rightarrow\RR$ is a smooth function, and we can access only the noisy function evaluations $F(x,\xi)$. A natural zeroth-order algorithm is to estimate the gradients of $f$ and use them in the context of Riemannian stochastic gradient descent. The main difficulty in doing so is the construction of the zeroth-order gradient estimation. Assuming that we have independent samples $u_i$ that are standard normal random vectors supported on $\T_{x}\M$, the tangent space at $x\in\mathcal{M}$,~\cite{li2022stochastic} proposed to construct the zeroth-order gradient estimator as
\begin{align}\label{zeroth_order_estimator}
    G^\Exp_{\mu}(x) = \frac{1}{m}\sum^m_{i=1}\frac{F(\Exp_{x}(\mu u_i),\xi_i) - F(x, \xi_i)}{\mu} u_i
\end{align}
where $\mu>0$ is a smoothing parameter. Note here that if a retraction is available, then one could also replace the exponential mapping with a retraction based estimator,
\begin{align}\label{zeroth_order_estimator_retr}
    G^\retr_{\mu}(x) = \frac{1}{m}\sum^m_{i=1}\frac{F(\retr_{x}(\mu u_i),\xi_i) - F(x, \xi_i)}{\mu} u_i.
\end{align}
The merit of having a Gaussian distribution on the tangent space is that the variance of the constructed estimator $G_{\mu}(x)$ will only depend on the intrinsic dimension $d$ of the manifold, and is independent of the dimension $n$ of the ambient Euclidean space. We refer to \cite{li2022stochastic} for the details of our zeroth-order estimator and its applications. See also~\cite{wang2021greene,wang2023sharp} for additional follow-up works.

To obtain an $\epsilon$-approximate stationary solution of~\eqref{stochastic_problem} (as in Definition~\ref{def:epsstat}) using the above approach, \cite{li2022stochastic} established a sample complexity of $\mathcal{O}(d/\epsilon^4)$, with $\mathcal{O}(1/\epsilon^2)$ iteration complexity and $m=\mathcal{O}(d/\epsilon^2)$ per-iteration batch size. Even considering $d=1$ for simplicity, this suggests for example that to get an accuracy of $\epsilon \approx 10^{-3}$, one needs batch-sizes of order $m\approx 10^{6}$ resulting in a  highly impractical per-iteration complexity. Intriguingly, when implementing these algorithms in practice, favorable results are obtained even when the batch-size is simply set between ten and fifty. Thus, there exists a discrepancy between the current theory and practice of stochastic zeroth-order Riemannian optimization. Furthermore, in online Riemannian optimization problems~\citep{maass2022tracking,wang2023online} where the data sequence is observed in a streaming fashion, waiting for very long time-periods in each iteration in order to obtain the required order of batch-sizes is highly undesirable.   

\begin{table}[t]
\begin{center}
\begin{small}
\begin{sc}
\begin{tabular}{|c|c|c|c|c|c|}
\hline
Result & Objective & Manifold & Operations & $m$ & N  \\
\hline
\thead{\texttt{Zo-RSGD}\\ \cite[Alg 1]{li2022stochastic}} & \thead{Smooth,\\ 2MB} & general & Retr & $\mathcal{O}(d/\epsilon^2)$ &$\Omega(1)$  \\  \hline

\multirow{2}{*}{\thead{\texttt{Zo-RASA},\\ Alg~\ref{algorithm2}, Thm \ref{theorem3}}} & \multirow{2}{*}{\thead{Smooth,\\ 2MB}} & \multirow{2}{*}{General} & \multirow{2}{*}{\thead{Exp map,\\ PT}} & $\mathcal{O}(d)$ & $\Omega(1)$  \\ \cline{5-6}
& & & & $\mathcal{O}(1)$ & $\Omega(d)$\\ \hline

\multirow{2}{*}{\thead{\texttt{Zo-RASA},\\ Alg \ref{algorithm2_vec_tran}, Thm \ref{theorem2_vec}}} & \multirow{2}{*}{\thead{Smooth,\\ 4MB}} & \multirow{2}{*}{\thead{Compact,\\2nd FF bound}} & \multirow{2}{*}{\thead{SO-Retr,\\ VT}} & $\mathcal{O}(d)$ & $\Omega(1)$\\ \cline{5-6}
& & & & $\mathcal{O}(1)$ &  $\Omega(1)$\\ \hline
\end{tabular}
\end{sc}
\end{small}
\end{center}
\caption{\textbf{Conditions required to establish a sample complexity of $\mathcal{O}(d/\epsilon^4)$ for various algorithms for convergence to stationarity in the sense of Definition~\ref{def:epsstat}}. For instance, to obtain the $\mathcal{O}(d/\epsilon^4)$ sample complexity for Alg~\ref{algorithm2}, we need to require $m=\mathcal{O}(d)$ and $N=\Omega(1)$, or $m=\mathcal{O}(1)$ and $N=\Omega(d)$. Here, 2MB and 4MB stand for bounded second central moment (i.e., variance) (Assumption \ref{assumption0_1}) and fourth central moment (Assumption \ref{assumption3}) respectively. 2nd FF stands for second fundamental form (Theorem \ref{thm_assump_vec_trans}) (see Section \ref{sec_manifold_basics} for definition of second fundamental form). SO-RETR stands for second-order retraction (Assumption \ref{assumption4}). PT and VT stand for parallel and vector transport respectively (see, Definition~\ref{def_vec_para_trans}). The parameter $d$ is the intrinsic dimension of the manifold $\mathcal{M}$, $m$ is the batch-size, $N$ is the total number of iterations required, and $\epsilon$ is the desired precision. Oracle complexity refers to the number of calls to the stochastic zeroth-order oracle. We also remark here that although \citet[Algorithm 1]{li2022stochastic} uses retraction, its convergence analysis also assumes retraction-based smoothness. For Zo-RASA, we need the initial batch-size $m_0=\mathcal{O}(d)$.}
\label{table0}
\end{table}

The main motivation of the current work stems from the above-mentioned undesirable issues associated with the use of mini-batches in stochastic Riemannian optimization algorithms by~\cite{li2022stochastic}. We address the problem by getting rid of the use of mini-batches altogether, and by developing batch-free, fully-online algorithm, Zeroth-order Riemannian Averaging Stochastic Approximation \texttt{(Zo-RASA)} algorithm,  for solving~\eqref{stochastic_problem}. We show that to obtain the sample complexity of $\mathcal{O}(d/\epsilon^4)$, \texttt{Zo-RASA} only requires $m=1$ (see the remark after Theorem \ref{theorem3}), which is a significant improvement compared to \cite{li2022stochastic}. The first version of \texttt{Zo-RASA} in Algorithm~\ref{algorithm2} uses exponential mapping and parallel-transports. However, this version is not implementation-friendly. As a case-in-point, consider the Stiefel manifold (see \eqref{stiefel}) for which we highlight that there is no closed-form expression for the parallel transport $P_{x^k}^{x^{k+1}}$. Indeed, they are only available as solutions to certain ordinary differential equation, which increases the per-iteration complexity of implementing Algorithm~\ref{algorithm2}. To overcome this issue and to develop a practical version of the RASA framework, we replace the exponential mapping and parallel transport by retraction and vector transport respectively, resulting in the  practical version of \texttt{Zo-RASA} method in Algorithm~\ref{algorithm2_vec_tran}. As we will discuss in Section \ref{sec_manifold_basics}, in the case of Stiefel manifolds, retractions cost only $1/4$ the time of an exponential mapping. Also, while there is no closed-form for parallel transport on Stiefel manifolds, vector transport has an easy closed-form implementation. We establish that Algorithm~\ref{algorithm2_vec_tran} has the same sample complexity as Algorithm~\ref{algorithm2}, with significantly improved per-iteration complexity. We now highlight two specific novelties that we introduce in this work to establish the above result.

\begin{itemize}[leftmargin=0.2in]
    \item \textbf{Moving-average gradient estimators and Lifting-based Riemannian-Lyapunov analysis.} We introduce a Riemannian moving-average technique (see, Line 4 in Algorithm~\ref{algorithm2} and Algorithm~\ref{algorithm2_vec_tran}) and a corresponding novel Riemannian-Lyapunov technique for analyzing zeroth-order stochastic Riemannian optimization problems, which works in the lifted space by tracking both the optimization trajectory and the gradient along the trajectory (see~\eqref{eta_def}). For Euclidean problems, these techniques were introduced and extended in~\cite{ruszczynski1983stochastic,ruszczynski1987linearization,ghadimi2020single,ruszczynski2021stochastic,balasubramanian2022stochastic}. However, those works rely heavily on the Euclidean structure. Non-trivial adaptions are needed to extend such methodology and analyses to the Riemannian settings; see Theorem~\ref{theorem3} and Theorem~\ref{theorem2_vec}.  \item \textbf{Approximation error between parallel and vector transports.} A major challenge in analyzing Algorithm~\ref{algorithm2_vec_tran} is to handle the additional errors introduced by the use of retractions and vector transports. We identify a novel geometric condition on the manifolds under consideration (see Assumption \ref{assumption1}) under which we provide novel error bounds between parallel and vector transports (see Theorem \ref{thm_assump_vec_trans}). We further show that the proposed condition, which plays a crucial role in our subsequent convergence analysis, is naturally satisfied if the \emph{second fundamental form} of the manifold is bounded. We remark that the obtained error bounds, between parallel and vector transport, are of independent interest and are potentially applicable to a variety of other Riemannian optimization problems.
\end{itemize}
In Table \ref{table0}, we summarize the sample complexities of stochastic zeroth-order Riemannian unioptimization algorithms

\subsection{Prior works} We refer to~\cite{absil2009optimization,boumal2020introduction} for a discussion on general Riemannian optimization methods. To the best of our knowledge,~\cite{li2022stochastic} provided the first oracle complexity results for zeroth-order stochastic Riemannian optimization. Following this,~\cite{wang2021greene, wang2023sharp, maass2022tracking} improved and extended the applicability of zeroth-order Riemannian optimization. A central concern in Riemannian optimization is the increased per-iteration complexity caused by the use of exponential mapping and (sometimes) parallel transport. To tackle this, retraction and vector transport are often preferred~\citep{absil2009optimization,boumal2020introduction}. Such replacements have thus far been considered in the deterministic settings, in the context of Riemannian quasi-Newton methods~\citep{huang2015broyden}, Riemannian variance reduction methods~\citep{sato2019riemannian}, Riemannian proximal gradient methods~\citep{chen2020proximal,huang2022riemannian} and Riemannian conjugate gradient methods~\citep{sato2022riemannian}. We discuss precise comparisons to this work later in Section~\ref{sec:comparisonprior}.

Stochastic gradient averaging methods in the Euclidean setting were studied in the several earlier works~\citep{polyak1977comparison,ruszczynski1983stochastic,xiao2009dual}. For nonconvex problems,~\cite{ghadimi2020single} analyzed the averaging stochastic approximation algorithm and established a sample complexity of $\mathcal{O}(1/\epsilon^4)$ to obtain an $\epsilon$-approximate first-order stationary solution without using mini-batches; see also~\cite{ghadimi2022stochastic} for a zeroth-order extension. For the smooth Riemannian setting,~\cite{han2020riemannian} used a related moving-average technique, and achieve $\mathcal{O}(\epsilon^{-3})$ sample complexity. However,~\cite{han2020riemannian} assumes a Lipschitz smooth-type inequality over $\grad F(x;\xi)$ itself under a given retraction (which is stronger than our assumption) and assume access to the computationally demanding isometric vector transport (see \eqref{eq_isometry_parallel_trans}). More importantly, they assume an opaque and rather strong condition that all iterates of their algorithm are close to a local optima of the problem to carry out their analysis.

\section{Basics of Riemannian optimization}\label{sec_manifold_basics} A differentiable manifold $\M$ is a Riemannian manifold if it is equipped with an inner product (called Riemannian metric) on the tangent space, $\langle \cdot, \cdot \rangle _x : \T_x\M \times \T_x\M \rightarrow \RR$, that varies smoothly on $\M$. The norm of a tangent vector is defined as $\|\xi\|_x\coloneqq\sqrt{\langle \xi, \xi\rangle _x}$. We drop the subscript $x$ and simply write $\langle \cdot, \cdot \rangle$ (and $\|\xi\|$) if $\M$ is an embedded submanifold with Euclidean metric. Here we use the notion of the tangent space $\T_{x}\M$ of a differentiable manifold $\M$, whose precise definition can be found in \citep[Chapter 8]{tu2011manifolds}. As an example, consider the Stiefel manifold given by
\begin{align}\label{stiefel}
\M= \St(n, p):=\{X\in\RR^{n\times p}: X^\top X=I_p\}.
\end{align}
The tangent space of $\St(n, p)$ is given by 
$\T_X\M=\{\xi\in\RR^{n\times p}: X^\top \xi+\xi^\top X=0\}.$ One could equip the tangent space with common inner product $\langle X, Y\rangle:=\tr(X^\top Y)$ to form a Riemannian manifold. For additional examples, see \citet[Chapter 3]{absil2009optimization}  or \citet[Chapter 7]{boumal2020introduction} .

We now introduce the concept of a Riemannian gradient and the notion of $\epsilon$-approximate first-order stationary solution for~\eqref{stochastic_problem}.
\begin{definition}[Riemannian Gradient]\label{def_riemann_grad}
    Suppose $f$ is a smooth function on Riemannian manifold $\M$. The Riemannian gradient $\grad f(x)$ is a vector in $\T_x\M$ satisfying $\left.\frac{d(f(\gamma(t)))}{d t}\right|_{t=0}=\langle v, \grad f(x)\rangle_{x}$ for any $v\in \T_x\M$, where $\gamma(t)$ is a curve satisfying $\gamma(0)=x$ and $\gamma'(0)=v$.
\end{definition}
\begin{definition}[$\epsilon$-approximate first-order stationary solution for \eqref{stochastic_problem}]\label{def:epsstat}
We call a point $\bar{x}$ an $\epsilon$-approximate first-order stationary solution for~\eqref{stochastic_problem} if it satisfies $\E[\| \grad f(\bar{x})\|_{\bar{x}}^2] \leq \epsilon^2$, where the expectation is with respect to both the problem and algorithm-based randomness. 
\end{definition} 

\textbf{Geodesics, retractions and exponential mappings.} Given two tangent vectors $\xi, \eta\in \T\M$, the Levi-Civita connection $\nabla:\T\M\times \T\M \rightarrow \T\M$, $(\xi,\eta)\rightarrow\nabla_{\xi}\eta\in \T\M$ is the ``directional differential" of $\eta$ along the direction of $\xi$, which is determined uniquely by the metric tensor $\langle\cdot, \cdot\rangle_x$. In Euclidean spaces, $\nabla_{\xi}\eta$ is just calculating the directional derivative of the vector field $\eta$ along $\xi$. For a Riemannian manifold $\M$, the geodesic $\gamma$ is a curve on $\M$ that satisfies $\nabla_{\gamma'} \gamma'=0$, i.e., the directional derivative along the tangent direction is always zero. Usually we find the geodesic with the initial value condition, $
\nabla_{\gamma'} \gamma'=0,\ \gamma(0)=x,\ \gamma'(0)=v,$ whose existence and uniqueness are locally guaranteed by the existence and uniqueness theorem for linear ODEs.

Given any curve $\gamma(t)$ on $\M$, one could calculate the length of the curve and define the distance between the two points $x,y\in\M$ respectively by $L(\gamma)\coloneqq\int_{a}^{b}\|\gamma'(t)\|_{\gamma(t)}dt\quad\text{and}\quad \mathsf{d}(x,y)\coloneqq\min_{\gamma,\gamma(a)=x,\gamma(b)=y}L(\gamma)$. If the manifold is a complete Riemannian manifold, according to \citep[Corollary 3.9]{do1992riemannian}, there exists a unique minimal geodesic $\gamma$ satisfying $\gamma(a)=x,\gamma(b)=y$ that minimizes $L(\gamma)$. Therefore, we can always calculate the distance with respect to the minimal geodesic as $    \mathsf{d}(x,y)=\int_{a}^{b}\|\gamma'(t)\|_{\gamma(t)}dt, \nabla_{\gamma'}\gamma'=0,\gamma(a)=x,\gamma(b)=y$,
which will be utilized in our error analysis in Section \ref{sec_avg_retr}.

A retraction mapping $\retr_x$ is a smooth mapping from $\T_x\M$ to $\M$ such that: $\retr_x(0)=x$, where $0$ is the zero element of $\T_x\M$, and the differential of $\retr_x$ at $0$ is an identity mapping, i.e., $\left.\frac{d \retr_x(t\eta)}{d t}\right|_{t=0}=\eta$, $\forall \eta\in \T_x\M$. In particular, the exponential mapping $\Exp_x$ on a Riemannian manifold is a retraction that generated by geodesics, i.e. $\Exp_{x}(t\xi)\coloneqq\gamma(t)$ where $\gamma$ is a geodesic with $\gamma(0)=x$ and $\gamma'(0)=\xi$. Notice that the retraction is not always injective from $\T_x\M$ to $\M$ for any point $x\in\M$, thus the existence of the inverse of the retraction function $\retr_{x}^{-1}$ is not guaranteed. However, when $\M$ is complete, the exponential mapping $\Exp_x$ is always defined for every $\xi\in \T_x\M$, and the inverse of the exponential mapping $\Exp_{x}^{-1}(y)\in \T_x\M$ is always well-defined for any $x,y\in\M$. Also, since $\Exp_{x}(t\xi)$ generates geodesics, we have $\mathsf{d}(x, \Exp_{x}(t\xi))=t\|\xi\|_x$. These are facts that we use in Assumption \ref{assumption0_2} and convergence proofs.

As an example, the retractions on Stiefel manifolds can be defined by the QR decomposition, $R_X(\xi):=Q$ where $X+\xi = QR$. It can also be defined through the Polar decomposition as $R_X(\xi):=U V^\top$, where $X+\xi = U\Sigma V^\top$ is the (thin) singular value decomposition of $X+\xi$. The geodesic on the Stiefel manifold is given by:
$
X(t)=\left[\begin{array}{ll}
X(0) & \dot{X}(0)
\end{array}\right] \exp \left(t\left[\begin{array}{cc}
A(0) & -S(0) \\
I & A(0)
\end{array}\right]\right)\left[\begin{array}{l}
I \\
0
\end{array}\right] \exp (-A(0) t),
$ for $A(t)=X^\top(t) \dot{X}(t)$ and $S(t)=\Dot{X}^\top(t)\dot{X}(t)$ with initial point $X(0)$ and initial speed $\dot{X}(0)$. The exponential mapping is thus given by $\Exp_{X(0)}(\dot{X}(0))=X(1)$. The computation cost of the QR and Polar decomposition retractions are of order $2dk^2+\mathcal{O}(k^3)$ and $3dk^2+\mathcal{O}(k^3)$, whereas as shown by \citet[Section 3]{chen2020proximal} the exponential mapping takes $8dk^2+\mathcal{O}(k^3)$, which illustrates the favorability of retractions in practical computations. We refer to \citet[Chapter 4]{absil2009optimization} and \citet[Chapter 3]{boumal2020introduction} for additional examples and more discussions on retractions and exponential mappings. 
\vspace{0.05in}

\textbf{Vector and parallel transport.} Vector transports are linear mappings from one tangent space to another, which can be formally defined below. 
\begin{definition}[Vector and parallel transport]\label{def_vec_para_trans}
    A vector transport $\mathcal{T}$ on a smooth manifold $\M$ is a smooth mapping
    $ \T\M\times \T\M\rightarrow \T\M:(\eta_x, \xi_x)\rightarrow \mathcal{T}_{\eta_x}(\xi_x)\in \T\M$, where the subscript $x$ means that the vector is in $\T_{x}\M$, such that: (i) There exists a retraction $R$ so that $\mathcal{T}_{\eta_x}(\xi_x)\in \T_{R_x(\eta_x)}\M$, (ii) $\mathcal{T}_{0_x}\xi_x = \xi_x$ for all $\xi_x\in \T_x\M$, and (iii) $\mathcal{T}_{\eta_x}(a\xi_x+b\zeta_x) = a\mathcal{T}_{\eta_x}(\xi_x)+b\mathcal{T}_{\eta_x}(\zeta_x)$, i.e., linearity.  Particularly, for a complete Riemannian manifold $(\M, \langle\cdot,\cdot\rangle)$, we can construct a special vector transport, namely the parallel transport $P$, that can map vectors to another tangent space ``parallelly'', i.e., $\forall \eta,\xi\in \T_x\M$ and $y\in\M$,
    \begin{align}\label{eq_isometry_parallel_trans}
        \langle P_{\Exp_{x}^{-1}(y)}(\eta), P_{\Exp_{x}^{-1}(y)}(\xi) \rangle_y=\langle \eta, \xi \rangle_x.
    \end{align}
    Notice that parallel transport is not the only transport that satisfies \eqref{eq_isometry_parallel_trans}, and we call the vector transport an isometric vector transport if it satisfies \eqref{eq_isometry_parallel_trans}.
    
\end{definition}
We can equivalently view $P$ as a mapping from the tangent space $\T_{x}\M$ to $\T_y\M$. We hence denote $P_{x}^{y}:\T_{x}\M\rightarrow \T_{y}\M$. Note that parallel transport depends on the curve along which the vectors are moving. If the curve is not specified, it refers to the case when we are considering the minimal geodesic connecting the two points, which exists due to completeness.

As an example, for the Stiefel manifold in~\eqref{stiefel}, there is no closed-form expression for the parallel transport, whereas one can always utilize the projection onto the tangent space, given by $\proj_{\T_X\M}(\xi) = (I-X X^\top)\xi + X\operatorname{skew}(X^\top \xi)$, where $\operatorname{skew}(A):=(A-A^\top)/2$, to transport $\xi\in \T_{X_0}\St(d, p)$ to $\T_{X}\St(d, p)$. We refer to \citet[Chapter 8]{absil2009optimization} and \citet[Chapter 10]{boumal2020introduction} for additional examples and more discussions on vector and parallel transports. \vspace{0.05in}

\textbf{Second fundamental form.} We now discuss the notion of second fundamental form, which will be helpful in characterizing a geometric condition used in Section \ref{sec_avg_retr} to quantify the error of approximating parallel transports with vector transports. In general, the notion of second fundamental form can be studied for general isometric immersions and we restrict here to the embedding in Euclidean spaces only for brevity. 
\begin{definition}[Second fundamental form]\label{def_2nd_fund_form}
  Suppose $\M\subset\RR^D$ is a complete Riemannian manifold equipped with the Euclidean metric. For any $\xi, \eta\in \T\M$, denote the extension of two vector fields to $\RR^D$ as $\bxi,\bet\in \RR^D$, also the directional derivative of $\Bar{\eta}$ along $\Bar{\xi}$ as $\Bar{\nabla}_{\bxi}\bet\in\RR^D$. The second fundamental form refers to the bilinear and symmetric vector, $
B(\xi,\eta) = \Bar{\nabla}_{\bxi}\bet - \nabla_{\xi}\eta \in (\T\M)^\bot,$ which quantifies the deviation of the Riemannian directional derivatives (depicted by Levi-Civita connection $\nabla$) from the Euclidean one (common directional derivative $\Bar{\nabla}$).   
\end{definition}
Finally, we remark that there are various definitions of second fundamental forms, among which the most common one is a quadratic form related to $B$; see~\citep[Chapter 6, Definition 2.2]{do1992riemannian}. Here we simply refer to $B$ as the second fundamental form.

\section{Zeroth-order RASA for smooth manifold optimization}\label{sec_smooth_avg}

We now introduce the Zeroth-order Riemannian Average Stochastic Approximation (\texttt{Zo-RASA}) algorithm for solving \eqref{stochastic_problem}. The formal procedure is stated in Algorithm~\ref{algorithm2}, where $P_{x}^{y}$ is the parallel transport from $\T_{x}\M$ to $\T_{y}\M$ along the minimum geodesic connecting $x$ and $y$. To establish the sample complexity of Algorithm~\ref{algorithm2}, we extend the analysis of~\cite{ghadimi2020single}, which is in-turn motivated by the lifting-technique introduced in~\cite{ruszczynski1983stochastic,ruszczynski1987linearization}, to the Riemannian setting. As such works heavily rely on the Euclidean structure, our proofs involve a non-trivial adaption of such techniques.

\begin{algorithm}[t]
   \caption{Zo-RASA}
   \label{algorithm2}
\begin{algorithmic}[1]
   \STATE {\bfseries Input:} Initial point $x^0\in\M$, $g^0=G_{\mu}^{\Exp}(x^0)$, total number of iterations $N$, parameters $\beta>0$, $\tau_0=1$, $\tau_k=1/\sqrt{N}$ or $\tau_k=1/\sqrt{dN}$ when $k\geq 1$, and stepsize $t_k=\tau_k/\beta$.
   \FOR{$k=0,1,2,\ldots,N-1$}
   \STATE $x^{k+1}\leftarrow \Exp_{x^{k}}(- t_k g^{k}) $
   \STATE $g^{k+1}\leftarrow (1-\tau_k) P_{x^{k}}^{x^{k+1}} g^{k} + \tau_k {P_{x^{k}}^{x^{k+1}} G_{\mu}^k}$ where $G_{\mu}^k=G_{\mu}^{\Exp}(x^k)$ is given by \eqref{zeroth_order_estimator} with batch-size $m=m_k$
   \ENDFOR
\end{algorithmic}
\end{algorithm}

In our convergence analysis, we always choose $\tau_0=1$, and we consider two choices of $\tau_k$ when $k\geq 1$: 
\begin{equation}\label{tau-2-choices}
\tau_k =1/\sqrt{N} \mbox{ or } \tau_k=1/\sqrt{d N}, k\geq 1,
\end{equation}
which corresponds to large or single batch, respectively. Moreover, we always choose $t_k=\tau_k/\beta$, where $\beta$ is a positive constant determined by the smoothness constant in Assumption \ref{assumption0_2} (see Theorem \ref{theorem3}), so that the step-size and the averaging weights are in the same order. Furthermore, we define 
    \begin{equation}\label{def-Gamma}
        \Gamma_{0}=\Gamma_{1} = 1, \mbox{ and } \Gamma_k=\Gamma_1\prod_{i=1}^{k-1}(1-\tau_i^2).
    \end{equation}
This leads to the following inequalities which will be used frequently in our convergence analysis:
\begin{equation}\label{tau-conditions}
    \sum_{i=k+1}^{N} \tau_{i}\Gamma_{i}\leq \Gamma_{k+1} \mbox{ and } \sum_{i=k+1}^{N} \tau_{i}^2\Gamma_{i}\leq \tau_k\Gamma_{k+1}.
\end{equation}

To proceed, we construct the following potential function
\begin{align}\label{eta_def}
W(x,g) \coloneqq (f(x) - f^*) - \eta(x,g),\quad\text{where}\quad \eta(x,g)\coloneqq-\frac{1}{2\beta}\|g\|_x^2,\ g\in \T_x\M,
\end{align}
where $f^*=\min_{x\in\M}f(x)$ and $\beta>0$ is a constant to be determined later. Note that the potential function in~\eqref{eta_def} has the component of both function value and the norm of the (estimated) gradients, also that $W$ is always non-negative. In our analysis, we proceed by bounding the difference of potential function between successive iterates. More specifically, using the convexity of the norm, for any pair $(x,g)$, we have $
\|\grad f(x)\|_{x}^2\leq -2\beta\,\eta(x,g) + 2\|g-\grad f(x)\|_{x}^2$. This observation will be leveraged in the proof of Theorem \ref{theorem3} to obtain the sample complexity of Algorithm~\ref{algorithm2} for obtaining an $\epsilon$-approximate stationary solution. 

We also highlight that our convergence analysis extensively utilizes the isometry property of parallel transport, stated in \eqref{eq_isometry_parallel_trans}, i.e., $\langle P_{x}^{y}(\eta), P_{x}^{y}(\xi) \rangle_y=\langle \eta, \xi \rangle_x$. This result is a generalization of the isometry in the Euclidean spaces, since the inner product in Euclidean spaces is unchanged if one moves the beginning point of the vectors together. A direct result of this identity is that the length of the vectors is unchanged, namely $\|P_{x}^{y}(\xi)\|_{y}=\|\xi\|_{x}$, which we will also use extensively.

We now introduce the assumptions needed for our analysis.
\begin{assumption}\label{assumption0_2}
   The function $f:\mathcal{M}\to\mathbb{R}$ is $L$-smooth on $\mathcal{M}$, i.e., $\forall x,y\in\mathcal{M}$, we have $\|P_{x}^{y}\grad f(x) - \grad f(y)\|_{y}\leq L\,\mathsf{d}(x,y)$. An immediate consequence (see, for example, \citet[Proposition 10.53]{boumal2020introduction}) of this condition is that we have  $|f(y)-f(x) - \langle \grad f(x), \Exp_{x}^{-1}(y) \rangle_{x}|\leq \frac{L}{2}\|\Exp_{x}^{-1}(y) \|_{x}^2$.
\end{assumption}

Assumption \ref{assumption0_2} is a generalization of the standard gradient-Lipschitz assumption in Euclidean optimization~\citep{nesterov2018lectures,lan2020first} to the Riemannian setting, and is made in several works~\citep{boumal2020introduction}. To generalize it to the Riemannian setting, due to the fact that $\grad f(x)$ and $\grad f(y)$ are not in the same tangent space, we need to utilize parallel transports $P_{x}^{y}$ to match the two vectors in the same tangent space. 

Throughout the paper, we define $\mathcal{F}_{k}$ as the $\sigma$-algebra generated by all the randomness till iteration $k$ of the algorithms. Namely, for Algorithm~\ref{algorithm2}, we have $\mathcal{F}_{k}=\sigma(\xi_0,\ldots,\xi_k,x_0,\ldots, x_k, g_0,\ldots, g_k)$.

\begin{assumption}\label{assumption0_1}
   Along the trajectory of the algorithm, the stochastic gradients are unbiased and have bounded-variance, i.e., for $k \in \{1,\ldots, N\}$, we have $\E_\xi [\grad F(x^k;\xi_k)|\mathcal{F}_{k-1}] = \grad f(x^k)$ and $\E_\xi[\|\grad F(x^k;\xi_k) - \grad f(x^k)\|_{x^{k}}^2|\mathcal{F}_{k-1}]\leq \sigma^2.$
\end{assumption}
The above assumption is widely used in stochastic Riemannian optimization literature; see, for example,~\cite{zhang2016riemannian,li2022stochastic,boumal2020introduction}, and generalizes the standard assumption used in Euclidean stochastic optimization~\citep{nesterov2018lectures,lan2020first}.

Now we proceed to the convergence analysis of Algorithm~\ref{algorithm2}. We first state the following standard result characterizing the approximation error of $G_{\mu}^{\Exp}$ (given by \eqref{zeroth_order_estimator}) to the true Riemannian gradient. 
\begin{lemma}[Proposition 1 in \cite{li2022stochastic} with exponential mapping]\label{bound_zo_estimator}
    Under Assumptions \ref{assumption0_2}, \ref{assumption0_1} we have $\|\E G_{\mu}^{\Exp}(x) - \grad f(x)\|_{x}^2 \leq \frac{\mu^2L^2}{4}(d+3)^3$, $\E\|G_{\mu}^{\Exp}(x)\|_{x}^2 \leq \mu^2 L^2(d + 6)^3 + 2(d+4) \|\grad f(x)\|_{x}^2$ and $\E\|G_{\mu}^{\Exp}(x) - \grad f(x)\|_{x}^2 \leq \mu^2 L^2(d + 6)^3 + \frac{8(d+4)}{m}\sigma^2+\frac{8(d+4)}{m}\|\grad f(x)\|_{x}^2$, where the expectation is taken toward all the Gaussian vectors in $G_{\mu}$ and the random variable $\xi$.
\end{lemma}

Based on the above result, we have the following Lemma \ref{lemma_zo_gk_gradk} which bounds the difference of $g^k$ to the true Riemannian gradient $\grad f(x^k)$, and Lemma \ref{lemma_zo_sum_gkplus1_gk} bounds the difference of two consecutive $g^k$, where we use parallel transport to make $g^k$ and $g^{k+1}$ in the same tangent space, i.e., $\|P_{x^{k+1}}^{x^{k}}g^{k+1} - g^k\|_{x^k}^2$.

\begin{lemma}\label{lemma_zo_gk_gradk}
    Suppose the Assumptions \ref{assumption0_2} and \ref{assumption0_1} hold, and $\{x^k,g^k\}$ is generated by Algorithm~\ref{algorithm2}. 
    We have
    \begin{align}\label{zo_gk_gradk_bound}
    & \E\|g^{k} - \grad f(x^{k})\|_{x^{k}}^2 \\ \leq & 
        \Gamma_{k}\tilde{\sigma}_0^2 + \Gamma_{k}\sum_{i=1}^{k}\Big( \frac{(1+\tau_{i-1})\tau_{i-1}}{\Gamma_{i}}\frac{L^2\|g^{i-1}\|_{x^{i-1}}^2}{\beta^2}+\frac{\tau_{i-1}^2}{\Gamma_{i}}\tilde{\sigma}_{i-1}^2+\tau_k\hat{\sigma}^2 \Big),\nonumber 
    \end{align}
where the expectation $\E$ is taken with respect to all random variables up to iteration $k$, including the random variables $\{u_i\}_{i=1}^k$ used to construct the zeroth-order estimator as in \eqref{zeroth_order_estimator}. Here the notations are defined as:
\begin{align}\label{def-tilde-sigma}
\begin{aligned}
    \hat{\sigma}^2&\coloneqq \frac{\mu^2L^2}{4}(d+3)^3\\
    \tilde{\sigma}_k^2&\coloneqq \sigma_{k}^2+\frac{8(d+4)}{m_k}\E\|\grad f(x^{k})\|_{x^k}^2\text{ where }\sigma_{k}^2\coloneqq\mu^2 L^2(d + 6)^3 + \frac{8(d+4)}{m_k}\sigma^2.
\end{aligned}
\end{align} 
Moreover, from \eqref{tau-conditions} we have 
    \begin{equation*}
    \begin{split}
        \sum_{k=1}^{N}\tau_k\E\|g^{k} - \grad f(x^{k})\|_{x^{k}}^2 \leq \sum_{k=0}^{N-1}\bigg( (1+\tau_{k})\tau_{k}\frac{L^2\E\|g^{k}\|_{x^{k}}^2}{\beta^2} + \tau_{k}^2\tilde{\sigma}_k^2+\tau_k\hat{\sigma}^2 \bigg)+\tilde{\sigma}_0^2, \\
        \sum_{k=1}^{N}\tau_k^2\E\|g^{k} - \grad f(x^{k})\|_{x^{k}}^2 \leq \sum_{k=0}^{N-1}\bigg( (1+\tau_{k})\tau^2_{k}\frac{L^2\E\|g^{k}\|_{x^{k}}^2}{\beta^2} + \tau_{k}^3\tilde{\sigma}_k^2+\tau^2_k\hat{\sigma}^2 \bigg)+ \sum_{k=1}^{N}\tau_k^2\tilde{\sigma}_0^2.
    \end{split}
    \end{equation*}
\end{lemma}

\begin{proof}
    Firstly, note that we have the following: $g^{k} - \grad f(x^{k}) =  (1-\tau_{k-1}) P_{x^{k-1}}^{x^{k}} g^{k-1} + \tau_{k-1} P_{x^{k-1}}^{x^{k}} G_{\mu}^{k-1} - \grad f(x^{k}) =(1-\tau_{k-1}) P_{x^{k-1}}^{x^{k}}(g^{k-1} - \grad f(x^{k-1})) + (P_{x^{k-1}}^{x^{k}}\grad f(x^{k-1}) - \grad f(x^{k}))+ \tau_{k-1} P_{x^{k-1}}^{x^{k}}(G_{\mu}^{k-1} - \grad f(x^{k-1}))=(1-\tau_{k-1}) P_{x^{k-1}}^{x^{k}}(g^{k-1} - \grad f(x^{k-1})) + \tau_{k-1} e_{k-1} + \tau_{k-1}\Delta_{k-1}^{f}$. Hence, we have 
    \begin{align}\label{zo_temp1}
    \begin{aligned}
    & \|g^{k} - \grad f(x^{k})\|_{x^{k}}^2 \\ 
    \leq &(1-\tau_{k-1})\|g^{k-1} - \grad f(x^{k-1})\|_{x^{k-1}}^2 + \tau_{k-1}\|e_{k-1}\|_{x^{k}}^2+\tau_{k-1}^2\|\Delta_{k-1}^{f}\|_{x^{k}}^2\\
        & + 2\tau_{k-1}\langle (1-\tau_{k-1})P_{x^{k-1}}^{x^{k}}(g^{k-1} - \grad f(x^{k-1})) + \tau_{k-1}e_{k-1}, \Delta_{k-1}^{f}\rangle_{x^{k}},
    \end{aligned}
    \end{align}
    where the notation is defined as $ e_{k-1} \coloneqq \frac{1}{\tau_{k-1}}(P_{x^{k-1}}^{x^{k}}\grad f(x^{k-1}) - \grad f(x^{k})),$ and $ \Delta_{k-1}^{f} \coloneqq P_{x^{k-1}}^{x^{k}} (G_{\mu}^{k-1} - \grad f(x^{k-1}))$. Denote $\delta_{k-1}=\langle (1-\tau_{k-1})P_{x^{k-1}}^{x^{k}}(g^{k-1} - \grad f(x^{k-1})) + \tau_{k-1}e_{k-1}, \Delta_{k-1}^{f}\rangle_{x^{k}}$. The main novelty in the proof of this lemma is that $\delta$ is no longer an unbiased estimator (which is true for the first-order situation). We have by Lemma \ref{bound_zo_estimator} that
    \begin{equation*}
    \begin{split}
        &2\E_{u^k}[\delta_{k-1}] = 2\langle (1-\tau_{k-1})P_{x^{k-1}}^{x^{k}}(g^{k-1} - \grad f(x^{k-1})) + \tau_{k-1}e_{k-1}, \E_{u^k}[\Delta_{k-1}^{f}|\mathcal{F}_{k-2}]\rangle_{x^{k}} \\
        \leq & \|(1-\tau_{k-1})P_{x^{k-1}}^{x^{k}}(g^{k-1} - \grad f(x^{k-1})) + \tau_{k-1}e_{k-1}\|_{x^{k} }^2 + \|\E_{u^k} G_{\mu}^{k-1} - \grad f(x^{k-1})\|_{x^{k-1}}^2 \\
        \leq & (1-\tau_{k-1})\|g^{k-1} - \grad f(x^{k-1})\|_{x^{k-1}}^2 + \tau_{k-1}\|e_{k-1}\|_{x^{k} }^2 + \hat{\sigma}^2.
    \end{split}
    \end{equation*}
    Notice that in the above computation, the expectation is only taken with respect to the Gaussian random variables that we used to construct $G_{\mu}(x^{k-1})$. Plugging this back to \eqref{zo_temp1}, we have  $ \E_{u^k}\|g^{k} - \grad f(x^{k})\|_{x^{k}}^2\leq  \tau_{k-1}\hat{\sigma}^2+ (1-\tau_{k-1}^2)\|g^{k-1} - \grad f(x^{k-1})\|_{x^{k-1}}^2 + \tau_{k-1}(1+\tau_{k-1})\|e_{k-1}\|_{x^{k}}^2 + \tau_{k-1}^2\|\Delta_{k-1}^{f}\|_{x^{k}}^2$.   Now dividing both sides of this inequality by our new definition of $\Gamma_k$, we get $\frac{1}{\Gamma_k}\E_{u^k}\|g^{k} - \grad f(x^{k})\|_{x^{k}}^2\leq \frac{1}{\Gamma_{k-1}}\|g^{k-1} - \grad f(x^{k-1})\|_{x^{k-1}}^2 + \frac{(1+\tau_{k-1})\tau_{k-1}}{\Gamma_k}\|e_{k-1}\|_{x^{k}}^2 + \frac{\tau_{k-1}^2}{\Gamma_k}\|\Delta_{k-1}^{f}\|_{x^{k}}^2 + \frac{\tau_{k-1}}{\Gamma_k}\hat{\sigma}^2$.

By Assumptions \ref{assumption0_2}, \ref{assumption0_1} and Lemma \ref{bound_zo_estimator}, we have that $\|e_{i}\|_{x^{i+1}}^2\leq \frac{L^2}{\tau_i^2}\mathsf{d}(x^i,x^{i+1})^2 \leq \frac{L^2 t_i^2\|g^i\|_{x^i}^2}{\tau_i^2} = \frac{L^2\|g^i\|_{x^i}^2}{\beta^2}$, and 
$ \E[\|\Delta_{i}^{f}\|_{x^{i+1}}^2|\mathcal{F}_{i-1}]\leq \sigma_i^2+\frac{8(d+4)}{m_i}\E[\|\grad f(x^i)\|_{x^i}^2|\mathcal{F}_{i-1}]$. Hence, by applying law of total expectation (to take the expectation over all random variables), we have $
        \frac{1}{\Gamma_k}\E\|g^{k} - \grad f(x^{k})\|_{x^{k}}^2
        \leq \frac{1}{\Gamma_{k-1}}\E\|g^{k-1} - \grad f(x^{k-1})\|_{x^{k-1}}^2 + \frac{(1+\tau_{k-1})\tau_{k-1}}{\Gamma_k}\frac{L^2\E\|g^{k-1}\|_{x^{k-1}}^2}{\beta^2} + \frac{\tau_{k-1}^2}{\Gamma_k}\Tilde{\sigma}_{k-1}^2 + \frac{\tau_{k-1}}{\Gamma_k}\hat{\sigma}^2$. Now by telescoping the sum in the above equation, we get (note that we take $g^0=G_{\mu}(x^0)$)
    \begin{align*}
    \begin{aligned}
        &\E\|g^{k} - \grad f(x^{k})\|_{x^{k}}^2\leq \Gamma_{k}\E\|G_{\mu}(x^0) - \grad f(x^{0})\|_{x^{0}}^2 \\
        &+ \Gamma_{k}\sum_{i=1}^{k}\bigg( \frac{(1+\tau_{i-1})\tau_{i-1}}{\Gamma_{i}}\frac{L^2\E\|g^{i-1}\|_{x^{i-1}}^2}{\beta^2} + \frac{\tau_{i-1}^2}{\Gamma_{i}}\Tilde{\sigma}_{i-1}^2 + \frac{\tau_{i-1}}{\Gamma_i}\hat{\sigma}^2 \bigg) \\
        &\leq \Gamma_{k}\tilde{\sigma}_0^2+ \Gamma_{k}\sum_{i=1}^{k}\bigg( \frac{(1+\tau_{i-1})\tau_{i-1}}{\Gamma_{i}}\frac{L^2\E\|g^{i-1}\|_{x^{i-1}}^2}{\beta^2} + \frac{\tau_{i-1}^2}{\Gamma_{i}}\Tilde{\sigma}_{i-1}^2 + \frac{\tau_{i-1}}{\Gamma_i}\hat{\sigma}^2 \bigg).
    \end{aligned}
    \end{align*}
    This proves \eqref{zo_gk_gradk_bound}. From \eqref{tau-conditions} we have
    \begin{align*}
    \begin{split}
        & \sum_{k=1}^{N}\tau_k\E\|g^{k} - \grad f(x^{k})\|_{x^{k}}^2 \\
        \leq &\sum_{k=1}^{N}\tau_k\Gamma_{k} \sum_{i=1}^{k}\bigg( \frac{(1+\tau_{i-1})\tau_{i-1}}{\Gamma_{i}}\frac{L^2\E\|g^{i-1}\|_{x^{i-1}}^2}{\beta^2} + \frac{\tau_{i-1}^2}{\Gamma_{i}}\Tilde{\sigma}_{i-1}^2 + \frac{\tau_{i-1}}{\Gamma_i}\hat{\sigma}^2 \bigg)+\tilde{\sigma}_0^2 \\
        = & \sum_{k=0}^{N-1}\bigg(\sum_{i=k+1}^{N}\tau_i\Gamma_{i}\bigg)\frac{1}{\Gamma_{k+1}}\bigg( (1+\tau_{k})\tau_{k}\frac{L^2\E\|g^{k}\|_{x^{k}}^2}{\beta^2} + \tau_{k}^2\Tilde{\sigma}_{k}^2 + \tau_{k}\hat{\sigma}^2 \bigg)+\tilde{\sigma}_0^2\\
        \leq & \sum_{k=0}^{N-1}\bigg( (1+\tau_{k})\tau_{k}\frac{L^2\E\|g^{k}\|_{x^{k}}^2}{\beta^2} + \tau_{k}^2\Tilde{\sigma}_{k}^2 +\tau_k\hat{\sigma}^2 \bigg)+\tilde{\sigma}_0^2,
    \end{split}
    \end{align*}
    where we used $\sum_{k=1}^{N}\tau_k\Gamma_{k}\leq\Gamma_{1} = 1$ 
    due to \eqref{tau-conditions}, so that the last term is simply $\tilde{\sigma}_0^2$. 

    By using similar calculations, we have that
    \begin{align*}
        &\sum_{k=1}^{N}\tau_k^2\E\|g^{k} - \grad f(x^{k})\|_{x^{k}}^2\leq\\
        &\sum_{k=1}^{N}\tau_k^2\Gamma_{k} \sum_{i=1}^{k}\bigg( \frac{(1+\tau_{i-1})\tau_{i-1}}{\Gamma_{i}}\frac{L^2\E\|g^{i-1}\|_{x^{i-1}}^2}{\beta^2}+ \frac{\tau_{i-1}^2}{\Gamma_{i}}\Tilde{\sigma}_{i-1}^2 + \frac{\tau_{i-1}}{\Gamma_{i}}\hat{\sigma}^2  \bigg)+ \sum_{k=1}^{N}\tau_k^2\tilde{\sigma}_0^2\\
        = & \sum_{k=0}^{N-1}\left(\sum_{i=k+1}^{N}\tau_i^2\Gamma_{i}\right)\frac{1}{\Gamma_{k+1}}\bigg( (1+\tau_{k})\tau_{k}\frac{L^2\E\|g^{k}\|_{x^{k}}^2}{\beta^2}+\tau_{k}^2\Tilde{\sigma}_{k}^2  + {\tau_{k}}\hat{\sigma}^2 \bigg)+ \sum_{k=1}^{N}\tau_k^2\tilde{\sigma}_0^2\\
        \leq & \sum_{k=0}^{N-1}\tau_k\bigg( (1+\tau_{k})\tau_{k}\frac{L^2\E\|g^{k}\|_{x^{k}}^2}{\beta^2} + \tau_{k}^2\Tilde{\sigma}_{k}^2+\tau_k\hat{\sigma}^2 \bigg)+ \sum_{k=1}^{N}\tau_k^2\tilde{\sigma}_0^2,
    \end{align*}
    which completes the proof.
\end{proof}

\begin{lemma}\label{lemma_zo_sum_gkplus1_gk}
    Suppose Assumptions \ref{assumption0_2} and \ref{assumption0_1} hold. We have
    \begin{align}\label{zo_sum_gkplus1_gk}
    \begin{split}
        \sum_{k=1}^{N}&\E\|P_{x^{k+1}}^{x^{k}}g^{k+1} - g^k\|_{x^k}^2\leq 2\sum_{k=0}^{N}\tau_k^2\hat{\sigma}^2 + 2\sum_{k=0}^{N}\left(\tau_k^2+\tau_k^3\right)\sigma_{k}^2+2\sum_{k=0}^{N}\tau_k^2\Tilde{\sigma}_0^2\\& +2\sum_{k=0}^{N}(1+\tau_{k})\tau_{k}^2\frac{L^2\E\|g^{k}\|_{x^{k}}^2}{\beta^2}+ 2\sum_{k=0}^{N}\left(\tau_k^2+\tau_k^3\right)\frac{8(d+4)}{m_k}\E\|\grad f(x^k)\|_{x^k}^2
    \end{split} 
    \end{align}
    where the expectation $\E$ is taken with respect to all random variables up to iteration $k$, which includes the Gaussian variables $u$ in the zeroth-order estimator as in \eqref{zeroth_order_estimator}.
\end{lemma}
\begin{proof}
    First note that $\|P_{x^{k+1}}^{x^{k}}g^{k+1} - g^k\|_{x^k}^2 = \tau_k^2\|G_{\mu}^k - g^k\|_{x^k}^2  = \tau_k^2\|G_{\mu}^k - \grad f(x^k) + \grad f(x^k) - g^k\|_{x^k}^2 
        \leq  2\tau_k^2\|G_{\mu}^k - \grad f(x^k)\|_{x^k}^2 + 2\tau_k^2\|\grad f(x^k) - g^k\|_{x^k}^2$.  Taking the expectation conditioned on $\mathcal{F}_{k-1}$, we get
    \begin{align*}
        &\frac{1}{2}\E[\|P_{x^{k+1}}^{x^{k}}g^{k+1} - g^k\|_{x^k}^2|\mathcal{F}_{k-1}]\\
        \leq & \tau_k^2\E[\|G_{\mu}^k - \grad f(x^k)\|_{x^k}^2|\mathcal{F}_{k-1}] + \tau_k^2\E[\|\grad f(x^k) - g^k\|_{x^k}^2|\mathcal{F}_{k-1}]\\
        \leq & \tau_k^2\bigg(\sigma_{k}^2+\frac{8(d+4)}{m_k}\E[\|\grad f(x^k)\|_{x^k}^2|\mathcal{F}_{k-1}]\bigg)+ \tau_k^2\E[\|\grad f(x^k) - g^k\|_{x^k}^2|\mathcal{F}_{k-1}],
    \end{align*}
where last inequality is by Lemma \ref{bound_zo_estimator}. Now using law of total expectation to take the expectation for all random variables and summing up over $k=0,...,N-1$, we have
    \begin{align*}
        & \frac{1}{2}\sum_{k=1}^{N}\E\|P_{x^{k+1}}^{x^{k}}g^{k+1} - g^k\|_{x^k}^2\\
        \leq & \sum_{k=1}^{N}\tau_k^2\sigma_{k}^2+ \sum_{k=1}^{N}\tau_k^2\frac{8(d+4)}{m_k}\E\|\grad f(x^k)\|_{x^k}^2 + \sum_{k=1}^{N}\tau_k^2\E\|\grad f(x^k) - g^k\|_{x^k}^2 \\
        \leq & \sum_{k=0}^{N}\tau_k^2\hat{\sigma}^2 + \sum_{k=0}^{N}\left(\tau_k^2+\tau_k^3\right)\sigma_{k}^2+\sum_{k=0}^{N}\tau_k^2\Tilde{\sigma}_0^2\\& +\sum_{k=0}^{N}(1+\tau_{k})\tau_{k}^2\frac{L^2\E\|g^{k}\|_{x^{k}}^2}{\beta^2}+ \sum_{k=0}^{N}\left(\tau_k^2+\tau_k^3\right)\frac{8(d+4)}{m_k}\E\|\grad f(x^k)\|_{x^k}^2,
    \end{align*}
    where the second inequality is by Lemma \ref{lemma_zo_gk_gradk}. 
\end{proof}

Now we are ready to present our main result. 
\begin{theorem}\label{theorem3}
    Suppose Assumptions \ref{assumption0_2} and \ref{assumption0_1} hold. In  Algorithm~\ref{algorithm2}, we set $
     \mu=\mathcal{O}\bigg(\frac{1}{L d^{3/2}N^{1/4}}\bigg)$, and $\beta\geq 4 L$. 
     Then the following holds.
    \begin{itemize}
        \item[(i)] If we choose $\tau_0=1$, $\tau_k= {1}/{\sqrt{N}}$, $k\geq 1$ and $m_k\equiv 8(d+4)$, $k\geq 0$, then we have $\frac{1}{N+1}\sum_{k=0}^{N} \E\|\grad f(x^{k})\|_{x^{k}}^2 \leq \mathcal{O}({1}/{\sqrt{N}})$.
        \item[(ii)] If we choose $\tau_0=1$, $\tau_k= {1}/{\sqrt{dN}}$, $k\geq 1$, $m_0=d$ and $m_k=1$ for $k\geq 1$, then we have $\frac{1}{N+1}\sum_{k=0}^{N} \E\|\grad f(x^{k})\|_{x^{k}}^2 \leq \mathcal{O}(\sqrt{{d}/{N}})$, for all $N = \Omega(d)$.
    \end{itemize}
Here the expectation $\E$ is taken with respect to all random variables up to iteration $k$, which includes the random variables $u$ in zeroth-order estimator \eqref{zeroth_order_estimator}. 

\end{theorem}

Before we proceed to the proof of Theorem \ref{theorem3}, we have the following Lemma \ref{lemma_zo_inequalities} which will be utilized in the proof.
\begin{lemma}\label{lemma_zo_inequalities}
    Suppose we take parameters the same as Theorem \ref{theorem3}, then we have
    \begin{subequations}
    \begin{gather}
        \frac{\tau_k}{2\beta} - \frac{\tau_k^2 L}{2 \beta^2} - \frac{(1+\tau_{k})\tau_{k}^2}{\beta}\frac{L^2}{\beta^2}\geq \frac{\tau_k}{4\beta}, \label{zo_temp_ineq1}\\
        \frac{\tau_k}{2} - \bigg( 4\bigg(\frac{2L^2}{\beta^2}+1\bigg)(1+\tau_k)+1 \bigg)\frac{8(d+4)}{m_k}\tau_k^2\geq \frac{\tau_k}{4}. \label{zo_temp_ineq2}
    \end{gather}
    \end{subequations}
    
\end{lemma}

\begin{proof}
    To show \eqref{zo_temp_ineq1}, using $\beta\geq 4L$, we just need to show that 
    $\tau_k/8+(1+\tau_k)\tau_k/16\leq 1/4$, which holds naturally in both cases (i) and (ii). 
    
    As for \eqref{zo_temp_ineq2}, again by $\beta\geq 4L$ we just need to show that $ \big( 4({1}/{8}+1)(1+\tau_k)+1 \big)({8(d+4)}/{m_k})\tau_k\leq {1}/{4}$. In case (i), this is equivalent to 
    $18\tau_k^2+22\tau_k-1\leq 0$, which is guaranteed when $N \geq 520$. For case (ii), similar calculation shows that we need $\tau_k\leq ({\sqrt{22^2+9/(d+4)} - 22})/{36}$, which is guaranteed when $N\geq 3.2\cdot 10^4\cdot {(d+4)^2}/{d}$.
\end{proof}

\begin{proof}[Proof of Theorem \ref{theorem3}]
By the isometry property of parallel transport, 
\begin{align*}
& \eta(x^{k},g^{k}) - \eta(x^{k+1},g^{k+1}) = \frac{1}{2\beta}\|g^{k+1}\|_{x^{k+1}}^2 - \frac{1}{2\beta}\|g^{k}\|_{x^{k}}^2  \\ = & \frac{1}{2\beta}\|P_{x^{k+1}}^{x^{k}}g^{k+1}\|_{x^{k}}^2 - \frac{1}{2\beta}\|g^{k}\|_{x^{k}}^2 \\ = &-\langle -\frac{1}{\beta}g^{k}, P_{x^{k+1}}^{x^{k}}g^{k+1} - g^k \rangle_{x^k} + \frac{1}{2\beta} \|P_{x^{k+1}}^{x^{k}}g^{k+1} - g^k\|_{x^k}^2.
\end{align*}

By combining this and Assumption \ref{assumption0_2}, we have the following bound for the difference of the merit function (defined in \eqref{eta_def}), evaluated at successive iterates:
\begin{equation*}
\begin{split}
    &W(x^{k+1}, g^{k+1}) - W(x^{k}, g^{k}) \\
    \leq & - t_k \langle \grad f(x^k),g^k \rangle_{x^{k}} + \frac{t_k^2 L}{2} \|g^k\|_{x^k}^2 + \frac{1}{\beta}\langle g^k, P_{x^{k+1}}^{x^{k}}g^{k+1} - g^k \rangle_{x^{k}} + \frac{1}{2\beta}\|P_{x^{k+1}}^{x^{k}}g^{k+1} - g^k\|_{x^k}^2 \\
    = & \bigg(\frac{ t_k^2 L}{2} -  t_k\bigg) \|g^k\|_{x^k}^2 + t_k\langle g^k, G_{\mu}^k - \grad f(x^k) \rangle_{x^{k}} + \frac{1}{2\beta}\|P_{x^{k+1}}^{x^{k}}g^{k+1} - g^k\|_{x^k}^2.
\end{split}
\end{equation*}
Moreover, we have 
\begin{align*}
    &\E_{u^k}[\langle g^k, G_{\mu}(x^k) - \grad f(x^k) \rangle_{x^k}] = \langle g^k, \E_{u^k} G_{\mu}(x^k) - \grad f(x^k) \rangle_{x^k} \\ \leq & \frac{1}{2}\|g^k\|_{x^k}^2 + \frac{1}{2}\|\E_{u^k} G_{\mu}(x^k) - \grad f(x^k)\|_{x^k}^2 \leq \frac{1}{2}\|g^k\|_{x^k}^2 + \frac{1}{2}\hat{\sigma}^2,
\end{align*}
where the expectation is only taken with respect to the Gaussian random variables that we used to construct $G_{\mu}(x^{k})$. Therefore, by using the law of total expectation, we have $ \E W(x^{k+1}, g^{k+1}) - \E W(x^{k}, g^{k}) \leq  \frac{1}{\beta}\left(\frac{\tau_k^2 L}{2\beta} - \frac{\tau_k}{2}\right) \E\|g^k\|_{x^k}^2 + \frac{\tau_k}{2\beta}\hat{\sigma}^2 + \frac{1}{2\beta}\E\|P_{x^{k+1}}^{x^{k}}g^{k+1} - g^k\|_{x^k}^2$, and we thus have (by summing up the above inequality over $k=0,...,N$):
    \begin{align}\label{zo_W_decrease}
    \begin{aligned}
        &\sum_{k=0}^{N}\left(\E W(x^{k+1}, g^{k+1}) - \E W(x^{k}, g^{k})\right) \\
        \leq & \sum_{k=0}^{N}\frac{1}{2\beta}\left(\frac{\tau_k^2 L}{\beta} - \tau_k\right) \E \|g^k\|_{x^k}^2 + \sum_{k=0}^{N}\frac{\tau_k}{2\beta}\hat{\sigma}^2 + \frac{1}{2\beta}\sum_{k=1}^{N}\E \|P_{x^{k+1}}^{x^{k}}g^{k+1} - g^k\|_{x^k}^2,
    \end{aligned}
    \end{align}
    where the last term sums from $1$ since $g^{1} - P_{x^{0}}^{x^{1}}g^0 = \tau_0(G_{\mu}^0-g^0)=0$.
    
    Utilizing \eqref{zo_sum_gkplus1_gk} and \eqref{zo_W_decrease}, we have (note that $W\geq 0$) 
    \begin{equation*}
    \begin{split}
        &\sum_{k=0}^{N}\frac{1}{2\beta}\left(\tau_k - \frac{\tau_k^2 L}{\beta}\right) \E \|g^k\|_{x^k}^2\leq W(x^0, g^0) + \sum_{k=0}^{N}\frac{\tau_k}{2\beta}\hat{\sigma}^2 + \frac{1}{2\beta}\sum_{k=1}^{N}\E \|P_{x^{k+1}}^{x^{k}}g^{k+1} - g^k\|_{x^k}^2 \\
        \leq & W(x^0, g^0) + \frac{1}{2\beta}\sum_{k=0}^{N}(\tau_k+2\tau_k^2)\hat{\sigma}^2 + \frac{1}{\beta}\sum_{k=0}^{N}\left(\tau_k^2+\tau_k^3\right)\sigma_{k}^2+\frac{1}{\beta}\sum_{k=0}^{N}\tau_k^2\Tilde{\sigma}_0^2\\& +\frac{1}{\beta}\sum_{k=0}^{N}(1+\tau_{k})\tau_{k}^2\frac{L^2\E\|g^{k}\|_{x^{k}}^2}{\beta^2}+ \frac{1}{\beta}\sum_{k=0}^{N}\left(\tau_k^2+\tau_k^3\right)\frac{8(d+4)}{m_k}\E\|\grad f(x^k)\|_{x^k}^2.
    \end{split}
    \end{equation*}
    Combining this with \eqref{zo_temp_ineq1} we have 
    \begin{align}\label{zo_temp5}
    \begin{aligned}
        \sum_{k=0}^{N}\tau_k& \E \|g^k\|_{x^k}^2
        \leq 4\beta W(x^0, g^0) + 2\sum_{k=0}^{N}(\tau_k+2\tau_k^2)\hat{\sigma}^2 + 4\sum_{k=0}^{N}\left(\tau_k^2+\tau_k^3\right)\sigma_{k}^2\\& +4\sum_{k=0}^{N}\tau_k^2\Tilde{\sigma}_0^2+4\sum_{k=0}^{N}\left(\tau_k^2+\tau_k^3\right)\frac{8(d+4)}{m_k}\E\|\grad f(x^k)\|_{x^k}^2.
    \end{aligned}
    \end{align}
    By Lemma~\ref{lemma_zo_gk_gradk} and \eqref{zo_temp5}, we get (also by $\tau_k\leq 1$)
    \begin{align}\label{zo_temp_ineq1.5}
    \begin{aligned}
        &\frac{1}{2}\sum_{k=0}^{N}\tau_k \E\|\grad f(x^{k})\|_{x^{k}}^2 \leq \sum_{k=0}^{N}\tau_k\E\|g^{k} - \grad f(x^{k})\|_{x^{k}}^2 + \sum_{k=0}^{N}\tau_k\E\|g^k\|_{x^{k}}^2 \\
        \leq & \sum_{k=0}^{N-1}\tau_{k}^2\tilde{\sigma}_k^2+\sum_{k=0}^{N-1}\tau_k\hat{\sigma}^2 + \left(\frac{2L^2}{\beta^2}+1\right)\sum_{k=0}^{N}\tau_k\E\|g^k\|_{x^{k}}^2+2\tilde{\sigma}_0^2 \\
        \leq & \left(\frac{8L^2}{\beta}+4\beta\right) W(x^0,g^0) + \sum_{k=0}^{N}\bigg[\tau_k+2\left(\frac{2L^2}{\beta^2}+1\right)(\tau_k+2\tau_k^2)\bigg]\hat{\sigma}^2\\
        &+\sum_{k=0}^{N}\bigg[ \tau_k^2+4\left(\frac{2L^2}{\beta^2}+1\right)(\tau_k^2+\tau_k^3) \bigg]\sigma_{k}^2 +\bigg[ 4\left(\frac{2L^2}{\beta^2}+1\right)\sum_{k=0}^{N}\tau_k^2 + 2 \bigg]\Tilde{\sigma}_0^2 \\
        &+ \sum_{k=0}^{N}\bigg[ 4\left(\frac{2L^2}{\beta^2}+1\right)(\tau_k^2+\tau_k^3)+\tau_k^2 \bigg]\frac{8(d+4)}{m_k}\E\|\grad f(x^k)\|_{x^k}^2,
    \end{aligned}
    \end{align}
    where $\tau_0\E\|g^{0} - \grad f(x^{0})\|_{x^{0}}^2\leq \Tilde{\sigma}_0^2$ is used in the last term on the second line. By combining \eqref{zo_temp_ineq1.5} and \eqref{zo_temp_ineq2} we get
    \begin{align}\label{zo_temp7}
    \begin{aligned}
    & \sum_{k=0}^{N}\frac{\tau_k}{4} \E\|\grad f(x^{k})\|_{x^{k}}^2
      \\   \leq
        &\sum_{k=0}^{N}\bigg[\frac{\tau_k}{2} - \bigg( 4\left(\frac{2L^2}{\beta^2}+1\right)(\tau_k^2+\tau_k^3)+\tau_k^2 \bigg)\frac{8(d+4)}{m_k}\bigg]\E\|\grad f(x^{k})\|_{x^{k}}^2 \\
        \leq & \left(\frac{8L^2}{\beta}+4\beta\right) W(x^0,g^0) + \sum_{k=0}^{N}\bigg[\tau_k+2\left(\frac{2L^2}{\beta^2}+1\right)(\tau_k+2\tau_k^2)\bigg]\hat{\sigma}^2\\
        &+\sum_{k=0}^{N}\bigg[ \tau_k^2+4\left(\frac{2L^2}{\beta^2}+1\right)(\tau_k^2+\tau_k^3) \bigg]\sigma_{k}^2 +\bigg[ 4\left(\frac{2L^2}{\beta^2}+1\right)\sum_{k=0}^{N}\tau_k^2 + 2 \bigg]\Tilde{\sigma}_0^2.
    \end{aligned}
    \end{align}

    For case (i) in Theorem~\ref{theorem3}, \eqref{zo_temp7} can be rewritten as 
    \[
    \frac{1}{N+1}\sum_{k=0}^{N} \E\|\grad f(x^{k})\|_{x^{k}}^2 
    \leq \frac{c_1W(x^0, g^0)}{\sqrt{N}}+ c_2\hat{\sigma}^2 + \frac{c_3\frac{1}{N}\sum_{k=0}^{N}\sigma_{k}^2}{\sqrt{N}} + \frac{c_4}{\sqrt{N}}\Tilde{\sigma}_0^2,
    \]
    for some absolute positive constants $c_1$, $c_2$, $c_3$ and $c_4$. The proof for case (i) is completed by noting that (see \eqref{def-tilde-sigma}) $\hat{\sigma}^2=\mathcal{O}(1/\sqrt{N})$, $\frac{1}{N}\sum_{k=0}^{N}\sigma_{k}^2=\mathcal{O}(1)$ and $\Tilde{\sigma}_0^2=\mathcal{O}(1)$. 
    
    For case (ii) in Theorem \ref{theorem3}, \eqref{zo_temp7} can be rewritten as 
    \[
    \frac{1}{N+1}\sum_{k=0}^{N} \E\|\grad f(x^{k})\|_{x^{k}}^2 
    \leq c_1' W(x^0, g^0)\sqrt{\frac{d}{N}}+ c_2'\hat{\sigma}^2 + \frac{c_3'\frac{1}{N}\sum_{k=0}^{N}\sigma_{k}^2}{\sqrt{d N}} + c_4'\sqrt{\frac{d}{N}}\Tilde{\sigma}_0^2,
    \]
    for some positive constants $c_1'$, $c_2'$, $c_3'$ and $c_4'$. The proof of case (ii) is completed by noting that $\Tilde{\sigma}_0^2=\mathcal{O}(1)$, $\hat{\sigma}^2=\mathcal{O}(1/\sqrt{N})$ and $\frac{1}{N}\sum_{k=0}^{N}\sigma_{k}^2=\mathcal{O}(d)$.
\end{proof}

\begin{remark}\label{sec:samplingtrick}
If we sample $R\in\{0,1,2,...,N\}$ with $
\mathbb{P}(R=k)=\tau_k/(\textstyle \sum_{k=0}^{N}\tau_k)$, then the left hand side of the inequalities in Theorem \ref{theorem3}, i.e., $\frac{1}{N+1}\sum_{k=0}^{N} \E\|\grad f(x^{k})\|_{x^{k}}^2$, becomes $\E\|\grad f(x^{R})\|_{x^{R}}^2$. If we use this sampling in case (i) of Theorem \ref{theorem3}, then to get an $\epsilon$-approximate stationary solution as in Definition~\ref{def:epsstat}, we require an iteration complexity of $N=\mathcal{O}(1/\epsilon^4)$ and so an oracle complexity of $Nm=\mathcal{O}(d/\epsilon^4)$. Case (i) requires $m=\mathcal{O}(d)$ per-iteration, which might be inconvenient in practice. Case (ii) of Theorem \ref{theorem3} avoids this, as in case (ii) both the iteration complexity and the oracle complexity are $N=\mathcal{O}(d/\epsilon^4)$, with batch size $m=\mathcal{O}(1)$. This makes case (ii) more convenient to use in practice, from a streaming or online perspective. For the simulations in Section~\ref{sec:experiments}, we thus choose $m=\mathcal{O}(1)$ and apply the result from case (ii). We also remark that the above results provide concrete solutions to the question raised by~\cite{scheinberg2022finite}, namely, on the need for mini-batches (and its order per-iteration) in zeroth-order stochastic optimization\footnote{Although~\cite{scheinberg2022finite} focuses on the Euclidean case, the discussion there also holds in the Riemannian setting.}.
\end{remark}

\begin{remark}\label{rmk_bdd_grad_no_d}
    Notice that to prove \eqref{zo_temp_ineq2}, we need $N=\Omega(d)$ for case (ii) in Theorem \ref{theorem3}. We can remove this condition if in addition we have that $\grad f(x)$ is uniformly upper bounded: $\|\grad f(x)\|_x\leq G,~ \forall x\in\M$; see also Assumption \ref{assumption_compactness} which we utilize in the next section. Under this condition, \eqref{zo_temp_ineq1.5} directly gives:
    \begin{align*}
    \begin{aligned}
        &\frac{1}{2}\sum_{k=0}^{N}\tau_k \E\|\grad f(x^{k})\|_{x^{k}}^2\leq  \sum_{k=0}^{N}\bigg[\tau_k+2\left(\frac{2L^2}{\beta^2}+1\right)(\tau_k+2\tau_k^2)\bigg]\hat{\sigma}^2\\
        &+\sum_{k=0}^{N}\bigg[ \tau_k^2+4\left(\frac{2L^2}{\beta^2}+1\right)(\tau_k^2+\tau_k^3) \bigg]\sigma_k^2 +\bigg[ 4\left(\frac{2L^2}{\beta^2}+1\right)\sum_{k=0}^{N}\tau_k^2 + 2 \bigg]\Tilde{\sigma}_0^2 \\
        &+ \sum_{k=0}^{N}\bigg[ 4\left(\frac{2L^2}{\beta^2}+1\right)(\tau_k^2+\tau_k^3)+\tau_k^2 \bigg]\frac{8(d+4)}{m_k}G^2+\left(\frac{8L^2}{\beta}+4\beta\right) W(x^0,g^0),
    \end{aligned}
    \end{align*}
whose right hand side has the same order as \eqref{zo_temp7}. Therefore in this case we do not need $N=\Omega(d)$ for case (ii) to achieve the same rates of convergence as in Theorem \ref{theorem3}.
\end{remark}

\section{RASA with retractions and vector transports}\label{sec_avg_retr}
Algorithm~\ref{algorithm2} is based on exponential mapping and parallel transport, which has a high per-iteration complexity for various manifold choices $\mathcal{M}$. In this section, we focus on reducing the per-iteration complexity of the \texttt{Zo-RASA} algorithm. The approach is based on replacing the exponential mapping and parallel transport with retractions and vector transports, respectively, which leads to practically efficient implementations and improved per-iteration complexity. 

The convergence analysis of algorithms with retractions and vector transports are sharply different and much harder than the one we presented in Section \ref{sec_smooth_avg}. Recall that the analysis in Section~\ref{sec_smooth_avg} relied on the isometry property \eqref{eq_isometry_parallel_trans} of the parallel transports, which is no longer available for vector transports. We hence assume explicit global error bounds between the difference of retraction to exponential mapping, as well as vector transport to parallel transport in Assumption \ref{assumption1}. In Section~\ref{sec:examples} we provide conditions on the manifold under which such assumptions are naturally satisfied and provide explicit examples. Based on this, we establish that under a bounded fourth (instead of the second) central moment condition, the same sample complexity result as in the previous section could be obtained for the practical versions of \texttt{Zo-RASA} algorithm on compact manifolds. 

\subsection{Approximation error of retractions and vector transports}\label{sec:errorbounds}

We start with the following condition on the vector transport used; recall the notation from Definition~\ref{def_vec_para_trans}. 
\begin{assumption}\label{assumption1}
    If $x^{+}=\retr_{x}(g)$, $g\in \T_{x}\M$, then with $\mathsf{d}$ denoting the geodesic distance, the vector transport $\mathcal{T}_{g}$ satisfies the following inequalities: 
    \begin{align}\label{eq_assump_vec_trans1}
        \|\mathcal{T}_{g}(v)\|_{x^{+}}\leq \|v\|_{x},\ \mathsf{d}(x, x^+)\leq \|g\|_x,\ \|\mathcal{T}_{g}(v) - P_{x}^{x^{+}}(v)\|_{x^{+}}\leq C \|v\|_{x}\mathsf{d}(x, x^+)
    \end{align}
    for any vector $v\in \T_{x}\M$. 
\end{assumption}

An intuitive explanation of the first inequality in \eqref{eq_assump_vec_trans1} is that our retraction and vector transport are ``conservative'' so that their length/magnitude is not longer than the exact operation of exponential mapping and parallel transport. As for the last inequality in \eqref{eq_assump_vec_trans1}, we are essentially positing that the vector transport would not ``twist'' the vector too much so that its difference from the parallel-transported vector is not large. In general, conditions in \eqref{eq_assump_vec_trans1} require the vector transport not to be very far from the parallel-transported vectors on the new tangent space. 

\subsubsection{Comparison to prior works}\label{sec:comparisonprior}

We now provide a detailed comparison to similar type of conditions proposed in two prior works,~\cite{huang2015broyden} and \cite{sato2022riemannian}, and highlight the differences and advantages of our proposal. According to the definition of vector transport in Definition \ref{def_vec_para_trans}, we need to specify a retraction associated with the transport so that $\mathcal{T}_{\eta_x}(\xi_x)\in \T_{R_x(\eta_x)}\M$. In this section, we consider the projection retraction, denoted simply as $R$.

Given two transports, $\mathcal{T}_{S}$ and $\mathcal{T}_{R}$, \cite{huang2015broyden} propose certain conditions on approximating one with the other. First they require that $\mathcal{T}_{S}$ is isometric, i.e., $
\langle\mathcal{T}_{S_{\eta}} (\xi), \mathcal{T}_{S_{\eta}} (\zeta)\rangle_{R_{x}(\eta)} = \langle\xi, \zeta\rangle_{x}
$, hence we can basically regard $\mathcal{T}_{S}$ as parallel transport for comparison. Let $\mathcal{T}_{R}$ denote the differential of the retraction, given by $
    \mathcal{T}_{R_\eta}(\xi)= D R_x(\eta)[\xi] = \frac{d}{d t}R_x(\eta + t\xi) \in \T_{R_x(\eta)}\M$. Now the conditions stated in Equations (2.5) and (2.6) in \cite{huang2015broyden} are as follows: there exists a \textit{neighborhood} $\mathcal{U}$ of $x$, such that $\forall y\in\mathcal{U}$ we have $\|\mathcal{T}_{S_\eta} - \mathcal{T}_{R_\eta}\|_{\mathrm{op}}\leq c_0\|\eta\|_x$ and $\|\mathcal{T}_{S_\eta}^{-1} - \mathcal{T}_{R_\eta}^{-1}\|_{\mathrm{op}}\leq c_0\|\eta\|_x$, where $\eta = R_{x}^{-1}(y)$ and $\|\cdot\|_{\mathrm{op}}$ is the operator norm. These assumptions are essentially local results, and as a result, \cite{huang2015broyden} needs to impose an additional stringent condition (see, their Assumption 3.2) that all the updates in their algorithms are already sufficiently close to the (local) optimal value to prove their convergence results. With the above conditions (in particular for a $\mathcal{T}_{1_\eta}$ satisfying their conditions in (2.5) and (2.6)), \cite{huang2015broyden} shows in Lemma 3.5 that \textit{locally} we have $\|\mathcal{T}_{1_\eta}(\xi) - \mathcal{T}_{2_\eta}(\xi)\|_{y}\leq c_0\|\eta\|_x\|\xi\|_x$. The proof of their Lemma 3.5 relies on the smoothness of the local coordinate form of the vector transports, which could hold only when we have a coordinate chart covering the local neighborhood we consider. Hence, the assumptions in~\cite{huang2015broyden} are in a different flavor from ours. In particular, our assumptions are global, and we show in Theorem \ref{thm_assump_vec_trans} that they are satisfied by a certain (global) assumption on the second fundamental form of the manifold $\M$.

    The existing work \cite{huang2015broyden} also assumes the so-called locking condition $\mathcal{T}_{S_\eta}(\xi) =\beta \mathcal{T}_{R_\eta}(\xi)$, where $\beta=\|\xi\|_x/\|\mathcal{T}_{R_\xi}(\xi)\|_{R_{\xi}(x)}$, which means that the approximating transport keeps the same direction as the parallel transport $\mathcal{T}_{S}$. In our analysis, we avoid such a condition since we are trying to transport two vectors $g^k$ and $G_{\mu}^k$ (see Algorithm \ref{algorithm2_vec_tran}), and not just one previous gradient as in the Riemannian quasi-Newton method \citep{huang2015broyden}. 
    Another existing work \cite{sato2022riemannian}  requires algorithm-specific conditions in their Assumption 3.1. To elaborate, we recall that the \textit{deterministic} Riemannian conjugate gradient iterates (Algorithm 1 in \cite{sato2022riemannian}) are given by $
    x_{k+1}\leftarrow R_{x_k}(t_k\eta_k)$ and $\eta_{k+1}\leftarrow-\grad f(x_{k+1})+\beta_{k+1} s_k \mathscr{T}^{k}(\eta_k),$ where $t_k$, $\beta_k$ and $s_k$ are parameters and $\mathscr{T}^{k}$ is a transport map from $\T_{x_k}\M$ to $\T_{x_{k+1}}\M$. Given this,  their Assumption 3.1 requires that there exist $C\geq 0$ and index sets $K_1\subset\mathbb{N}$ and $K_2 = \mathbb{N}-K_1$ such that $ \left\|\mathscr{T}^{(k)}\left(\eta_k\right)-\mathrm{D} R_{x_k}\left(t_k \eta_k\right)\left[\eta_k\right]\right\|_{x_{k+1}} \leq C t_k\left\|\eta_k\right\|_{x_k}^2, k \in K_1$
and $\left\|\mathscr{T}^{(k)}\left(\eta_k\right)-\mathrm{D} R_{x_k}\left(t_k \eta_k\right)\left[\eta_k\right]\right\|_{x_{k+1}} \leq C\left(t_k+t_k^2\right)\left\|\eta_k\right\|_{x_k}^2,  k \in K_2$.

Our assumption differs from the above in three aspects: (i) we do not make algorithm-specific assumptions, where each inequality depends on the iterate number $k$; (ii) we are not only comparing transporting $\eta_k$ (which is the direction along which we update $x^k$), but also the zeroth-order estimator $G_{\mu}^k$ (see Algorithm \ref{algorithm2_vec_tran}), i.e., we assume a more general inequality by replacing $\mathrm{D} R_{x}(t_k \eta)[\eta]$ with $\mathrm{D} R_{x}(t_k \eta)[\xi]$, where $\xi$ can be different from $\eta$; (iii) we \emph{derive} the last inequality in \eqref{eq_assump_vec_trans1} using global assumption of second fundamental form of the manifold $\M$ in Theorem \ref{thm_assump_vec_trans}, instead of \emph{assuming} it.

\subsubsection{Illustrative Examples}\label{sec:examples}

We now further inspect Assumption \ref{assumption1} by checking the conditions under which \eqref{eq_assump_vec_trans1} holds in general, and also verifying it for various matrix-manifolds arising in applications.

We start with the first inequality in \eqref{eq_assump_vec_trans1}. It holds naturally if the manifold is a submanifold and the vector transport is the orthogonal projection, due to the non-expansiveness of orthogonal projections. The second inequality in \eqref{eq_assump_vec_trans1} is much trickier. For the scope of this work, we show that the second equation in \eqref{eq_assump_vec_trans1} holds for projectional retractions and projectional vector transports on Stiefel manifold, which also includes spheres and orthogonal groups as special cases. If the inverse of the retraction in Assumption \ref{assumption1} is well-defined, the second inequality in \eqref{eq_assump_vec_trans1} could equivalently be stated as
$\|\Exp_{x}^{-1}(x^+)\|_{x}\leq \|\retr_{x}^{-1}(x^+)\|_x$, which may hold for a larger class of manifolds and retractions. We leave a detailed study of this as future work.  

\textbf{Stiefel manifold.} Consider the Stiefel manifold $\St(d, p)$ defined in~\eqref{stiefel}, with the tangent space $\T_{X}\St(d, p)=\{\xi|X^\top \xi+\xi^\top X=0\}$ and Euclidean inner product $\langle X, Y\rangle:=\tr(X^\top Y)$. We consider the projectional retraction~\citep{absil2012projection} given by $X^+=R_X(\xi):=U V^\top$, where $X+\xi = U\Sigma V^\top$ is the (thin) singular value decomposition of $X+\xi$. Also, the projectional vector transport $\mathcal{T}$ is simply projecting a tangent vector $\xi\in \T_{X_0}\St(d, p)$ to $\T_{X}\St(d, p)$. It is clear that $\|\mathcal{T}(\xi)\|\leq \|\xi\|$ due to the non-expansiveness of orthogonal projections (note that $\T_{X}\St(d, p)$ is simply a linear subspace). To show $\mathsf{d}(X, X^+)\leq \|\xi\|$, denote $\gamma(t)$ the minimal geodesic connecting $X$ and $X^+$ with $\gamma(0)=X$ and $\gamma(1)=X^+$, so that $\mathsf{d}(X, X^+) = \int_{0}^{1}\|\gamma'(t)\| d t$. Notice that we can define another curve $c(t)=U(t)V^\top(t)$, where $X+t \xi=U(t)\Sigma(t)V^\top(t)$ is the singular value decomposition. The curve $c(t)=\retr_{X}(t\xi)$ is the parameterized curve of projectional retraction. Now using the distance with respect to the minimal geodesic, we have $\mathsf{d}(X, X^+) = \int_{0}^{1}\|\gamma'(t)\| d t \leq \int_{0}^{1}\|c'(t)\| d t \leq \int_{0}^{1}\|\xi\| d t = \|\xi\|$, where $\|c'(t)\|\leq \|\xi\|$ is due to the non-expansiveness of orthogonal projections, namely, $  \|c(t_1)-c(t_2)\|\leq \|X+t_1\xi - (X+t_2\xi)\|$. Indeed, although $\St(d, p)$ is not a convex set, the non-expansiveness condition still holds~\citep{gallivan2010note}, because $(X+\xi)^\top(X+\xi)=I_p+\xi^\top\xi\succeq I_p$, and the projection of $X+\xi$ onto the Stiefel manifold is the same as projection onto its convex hull $\{X\in\RR^{d\times p}|\|X\|_2\leq 1\}$. Now we turn to the last inequality in \eqref{eq_assump_vec_trans1}. Given a complete embedded submanifold, we can show that the last inequality in \eqref{eq_assump_vec_trans1} holds under the boundedness of the second fundamental form in Theorem \ref{thm_assump_vec_trans}, given that the vector transport is the orthogonal projection to the new tangent space.

\begin{theorem}\label{thm_assump_vec_trans}
    Suppose $\M$ is an embedded complete Riemannian submanifold of Euclidean space. Suppose for all unit vector $\xi,\eta\in \T\M$, $\|\xi\|=\|\eta\|=1$, the norm of the second fundamental form $B(\xi,\eta)$ is bounded by constant $C$. Consider the parallel transport $P_{x}^{y}$ along the minimal geodesic from $x\in\M$ to $y\in\M$, we have $\|\proj_{\T_{y}\M}(v) - P_{x}^{y}(v)\|\leq C \|v\|\mathsf{d}(x,y)$, for any $v\in \T_{x}\M$. That is, the last inequality in \eqref{eq_assump_vec_trans1} holds with constant $C$.
\end{theorem}

\begin{proof}
    Without loss of generality, we assume $\|v\|=1$, otherwise conduct the proof for $v/\|v\|$. Denote the minimum geodesic $\gamma$ with unit speed connecting $x$ and $y$, parameterized by variable $t$, also denote the parallel transported vector of $v$ along $\gamma$ as $v(t)$, i.e. $v(0)=v$. Now for the extrinsic geometry, we denote $v=v^{\top}(t)+v^{\bot}(t)$, where $v^{\top}(t)\in \T_{\gamma(t)}\M$ and $v^{\bot}(t)$ is orthogonal to $\T_{\gamma(t)}\M$. Note that the left-hand side of the inequality we want to prove is now parameterized as $\|v(t) - v^{\top}(t)\|$.

    Now since $v(t)$ is a parallel transport of $v$, the tangent component must be zero, i.e., $(v'(t))^\top = 0$. Now consider any parallel unit vector $z(t)\in \T_{\gamma(t)}\M$ along $\gamma$, then $ \langle (v^\bot)'(t), z(t) \rangle = -\langle v^\bot(t), z'(t) \rangle= -\langle v^\bot(t), B(\gamma'(t),z(t)) \rangle,$    where $B$ is the second fundamental form. Along with the fact that $(v^\top)'=-(v^\bot)'$ we get $
    \langle (v^\top)'(t), z(t) \rangle =\langle v^\bot(t), B(\gamma'(t),z(t)) \rangle$.  Now the right-hand side has a uniform upper bound of $C$, and by the arbitrarily chosen $z(t)\in \T_{\gamma(t)}\M$, we get $
        \|((v^\top)'(t))^\top\|\leq C$.
 
    We can now bound the derivative of $\|v(t) - v^{\top}(t)\|$ as $(\|v(t) - v^{\top}(t)\|^2)' = (1 - 2\langle v(t), v^{\top}(t)\rangle + \|v^{\top}(t)\|^2)' = - 2\langle v(t), (v^{\top}(t))'\rangle + 2\langle v^{\top}(t), (v^{\top}(t))'\rangle  = 2\langle v^{\top}(t) - v(t), ((v^{\top}(t))')^\top\rangle\leq 2 C\|v^{\top}(t) - v(t)\|.$ Therefore, we get $\|v(t) - v^{\top}(t)\|'\leq C$. Now integrating the above inequality from $x$ to $y$ along the minimal geodesic $\gamma$ (i.e., with respect to $t$) and using the distance with respect to the minimal geodesic, we obtain
$\|\proj_{\T_{y}\M}(v) - P_{x}^{y}(v)\|\leq C \mathsf{d}(x, y)$, which completes the proof.
\end{proof}

Theorem \ref{thm_assump_vec_trans} connects extrinsic and intrinsic geometry by measuring the difference of orthogonal projection (extrinsic operation) and parallel transport (intrinsic operation), which might be of independent interest for studying embedded submanifolds. The condition in Theorem \ref{thm_assump_vec_trans} is stronger than the bounded sectional curvature condition since if the second fundamental form is bounded, the sectional curvature is also bounded by the Gauss formula (see Chapter 6, Theorem 2.5 in \cite{do1992riemannian}). We point out that the condition of Theorem \ref{thm_assump_vec_trans} is still satisfied by all the embedded submanifold applications we consider, namely the sphere, the orthogonal group and the Stiefel manifold. In particular, we have the following observation.

\begin{proposition}\label{thm_bound_2ndff_compact}
    Suppose $\M$ is a compact complete embedded Riemannian submanifold of Euclidean space (i.e. satisfying Assumption \ref{assumption_compactness}), then the norm of the second fundamental form $\|B(\xi,\eta)\|$ is uniformly bounded for all unit vector $\xi,\eta\in \T\M$, $\|\xi\|=\|\eta\|=1$.
\end{proposition}
The proof is immediate, since for all unit vector $\xi,\eta\in \T\M$, $\|B(\xi,\eta)\|\in\RR$ is a smooth function defined over a compact domain, and therefore it is upper bounded. As a result, Assumption \ref{assumption1} holds for all the embedded submanifold applications we consider, namely the sphere, the orthogonal group and the Stiefel manifold. 

\begin{remark}
We remind the readers that Theorem \ref{thm_assump_vec_trans} requires the  embedded submanifold assumption, yet Assumption \ref{assumption1} does not, as long as \eqref{eq_assump_vec_trans1} hold. This is also the main reason why we summarize our assumption as in Assumption \ref{assumption1}, and not present Theorem \ref{thm_assump_vec_trans} directly.
\end{remark}

\textbf{Example: Grassmann manifold.} Above, we have shown that Assumption \ref{assumption1} holds for a class of embedded matrix submanifolds. Yet another setting is that of quotient manifolds (e.g., the Grassmann manifold) which arises in applications of Riemannian optimization. Such manifolds are not naturally embedded submanifolds of a Euclidean space. As a result, we can inspect Assumption \ref{assumption1} directly for such manifolds. Taking the Grassmann manifold as an example, we next verify Assumption \ref{assumption1}. To proceed, we utilize the following result.

\begin{lemma}\label{lemma_principal_angle}
    Suppose $X\in \St(d, p)$, $G\in\RR^{d\times p}$ with $X^\top G=0$, and the QR decomposition of $X+G=Q R$ where $Q\in\St(d, p)$ and $R\in\RR^{p\times p}$ is upper triangular. The principal angle between the subspace spanned by $X$ and $Q$ is given by $\|\Theta\|_F$, where $\Theta:=\arccos(\Sigma)$ where $\Sigma$ is the singular value matrix of $X^\top Q$, i.e., $X^\top Q=U \Sigma V^\top $; see, for example \citet[Section 4.3]{edelman1998geometry}. We have that $\|\Theta\|_F\leq \|G\|_F.$
\end{lemma}

\begin{proof}
    Since $R^\top R = (X+G)^\top(X+G) = I_p+\|G\|_F^2$, we know that all the singular values of $R$ are greater than or equal to $1$. Denote $\Sigma=\diag([\sigma_1,...,\sigma_p])$. Since $X^\top Q = X^\top(X+G) R^{-1}=R^{-1}$, we know that the singular value decomposition of $R=V \Sigma^{-1} U^\top$ (which implies that $\sigma_i\leq 1$, $\forall i=1,2,....,p$) and $\|R\|_F^2= \|\Sigma^{-1}\|_F^2 = \sum_{i=1}^{p}\frac{1}{\sigma_i^2}$. Also, as $\|R\|_F^2=\|X+G\|_F^2=\tr((X+G)^\top(X+G))=p+\|G\|_F^2,$ we get $\|G\|_F^2 = \sum_{i=1}^{p}\frac{1}{\sigma_i^2} - p$. Thus, $
        \|\Theta\|_F^2 = \|\arccos(\Sigma)\|_F^2=\sum_{i=1}^{p}(\arccos(\sigma_i))^2 
        \leq  \sum_{i=1}^{p} (\frac{1}{\sigma_i^2} - 1)= \|G\|_F^2$, where we use the fact that $(\arccos(t))^2\leq \frac{1}{t^2}-1$, $\forall t\in(0, 1]$.
\end{proof}

Now we can inspect the Grassmann manifold. The Grassmann manifold $\Gr(d, p)$ is the set of all $p$-dimensional subspace of $\RR^d$; see, for example,~\cite[Section 2.1]{absil2009optimization}. A quotient formulation writes $\Gr(d, p)=\St(d, p)/\mathcal{O}(p)$ with $\mathcal{O}(p)=\{Q\in\RR^{p\times p}|Q^\top Q=I_p\}$ being the orthogonal group. The elements of the Grassmann manifold can be expressed as $[X]\in\Gr(d,p)$ with $[X]:=\{XQ|Q\in\mathcal{O}(p)\}$ and $X\in\St(d, p)$. The element $\bar{\xi}$ on the tangent space $\T_{[X]}\Gr(d, p)$ can be shown with a one-to-one mapping (called the horizontal lift) to the set $[\xi]$ with $\xi\in \T_{X}\St(d, p)$ and $X^\top\xi=0$. 

    Suppose we start from an element $[X]\in\Gr(d, p)$ with $X\in\St(d, p)$ and the initial speed $\bar{G}\in \T_{[X]}\Gr(d, p)$, where $G\in \T_{X}\St(d, p)$ and $X^\top G=0$. We denote the singular value decomposition of $G=U\Sigma V^\top$ with $U\in\RR^{d\times p}$ and $\Sigma, V\in\RR^{p\times p}$. Then the exponential mapping is given by $
        Y := \Exp_{[X]}(\bar{G})=[X V\cos(\Sigma)+U\sin(\Sigma)]$,
    where $\sin$ and $\cos$ are matrix trigonometric functions; see~\cite[Example 5.4.3]{absil2009optimization}. Also, the parallel transport is given by: $\Bar{\xi_1}=P_{[X]}^{[Y]}(\bar{\xi})$ with $ \xi_1=-X V\sin(\Sigma)U^\top \xi + U \cos(\Sigma)U^\top \xi + (I-U U^\top) \xi$. See \cite[Example 8.1.3]{absil2009optimization}. Hence, the projectional retraction is given by $
        Y' := \retr_{[X]}(\bar{G})=[X + G] = [Q],$ where $X+G = QR$ is the QR decomposition of $X+G$; see~\cite[Example 4.1.5]{absil2009optimization}. Furthermore, the projectional vector transport is given by $\Bar{\xi_2}=\mathcal{T}_{\Bar{G}}(\bar{\xi})$ with $
        \xi_2=(I - Y Y^\top)\xi$. See \cite[Example 8.1.10]{absil2009optimization}. 
        
    Now we show that \eqref{eq_assump_vec_trans1} is satisfied. It is obvious that $\|\mathcal{T}_{\Bar{G}}(\bar{\xi})\| = \|(I - Y Y^\top)\xi\|\leq \|\xi\|$. The geodesic distance of $[X]$ and the projectional retraction $[Q]$ is exactly the principal angle between the subspace spanned by $X$ and $Q$, see \cite[Section 4.3]{edelman1998geometry}. Following Lemma \ref{lemma_principal_angle}, we can hence conclude that $\mathsf{d}([X], [Q]) = \|\Theta\|_F\leq \|G\|_F$. Now we inspect the last equation in \eqref{eq_assump_vec_trans1}. We can directly check that $
        \|\xi_1 - \xi_2\|_{F}=\|A\xi\|_{F}\leq \|A\|_{F}\|\xi\|_{F}$,    with
    \begin{align*}
        A\coloneqq& -X V\sin(\Sigma)U^\top + U \cos(\Sigma)U^\top + Y Y^\top-U U^\top \\
        =&-X V\sin(\Sigma)U^\top + U \cos(\Sigma)U^\top - U \cos^2(\Sigma)U^\top + X V\cos^2(\Sigma) V^\top X^\top \\&+ U \sin(\Sigma)\cos(\Sigma)V^\top X^\top + X V \cos(\Sigma)\sin(\Sigma)U^\top.
    \end{align*}
    Note also that we have the bound 
    \begin{align*}
        \|A\|=&\|-X V\sin(\Sigma)U^\top + U \cos(\Sigma)U^\top - U \cos^2(\Sigma)U^\top + X V\cos^2(\Sigma) V^\top X^\top \\&+ U \sin(\Sigma)\cos(\Sigma)V^\top X^\top + X V \cos(\Sigma)\sin(\Sigma)U^\top\| \\
        \leq & \|\sin(\Sigma)\| + \|\cos(\Sigma)(I - \cos(\Sigma))\| + 2\|\sin(\Sigma)\cos(\Sigma)\| \leq 4\|\sin(\Sigma)\|\leq 4\|G\|,
    \end{align*}
    where we use the fact that $X^\top X=U^\top U=V^\top V=I_p$ and all norms are the Frobenius norm. Therefore, we see that the last equation in \eqref{eq_assump_vec_trans1} is satisfied with $C=4$.

\subsection{Convergence of retraction and vector transport based \texttt{Zo-RASA}}\label{sec:pracconvergence}

We now proceed to the convergence analysis of \texttt{Zo-RASA} algorithm with retraction and vector transports. Algorithm \ref{algorithm2_vec_tran} is the analog of Algorithm~\ref{algorithm2}, using retraction and vector transport. Notice that the zeroth-order estimator used in Algorithm \ref{algorithm2_vec_tran} is as defined in \eqref{zeroth_order_estimator_retr}, which is with respect to the retraction in contrast to \eqref{zeroth_order_estimator}. Also $\mathcal{T}$ is the vector transport where we write $\mathcal{T}^k:=\mathcal{T}_{- t_k g^{k}}$ for brevity. The vector transport we use in experiments is simply the orthogonal projection onto the target tangent space.

For our analysis, apart from the smoothness condition in Assumption~\ref{assumption0_2}, we also need to assume that the manifold is compact.
\begin{assumption}\label{assumption_compactness}
The manifold $\M$ is compact with diameter $D$, and the Riemannian gradient satisfies $\|\grad f(x)\|_x\leq G$. 
\end{assumption}
Here, $G$ could potentially be a function of $D$ and the constant $L$ from Assumption \ref{assumption0_2}, due to compactness and smoothness. 
We remark that this compactness assumption is satisfied by various matrix manifolds like the Stiefel manifold and the Grassmann manifold (see, for example, Lemma 5.1 in \cite{milnor1974characteristic}).

\begin{algorithm}[t]
   \caption{Zo-RASA with retraction and vector transport}
   \label{algorithm2_vec_tran}
\begin{algorithmic}[1]
   \STATE Change the updates of $x^{k+1}$ and $g^{k+1}$ in Algorithm~\ref{algorithm2} respectively to
   \begin{align*}
   x^{k+1}&\leftarrow \retr_{x^{k}}(- t_k g^{k}) \quad\text{and}\quad
g^{k+1}\leftarrow (1-\tau_k) \mathcal{T}^k (g^{k}) + \tau_k \mathcal{T}^k (G_{\mu}^k),
\end{align*}
 where $G_{\mu}^k=G_{\mu}^{\retr}(x^k)$ is given by \eqref{zeroth_order_estimator_retr} with batch-size $m=m_k$.
\end{algorithmic}
\end{algorithm}

Turning to the stochastic gradient oracles, the bounded second moment condition in Assumption~\ref{assumption0_1} is now replaced by the following condition of bounded fourth central moment. Such a condition is needed to conduct our convergence analysis. It is interesting to relax this assumption or show this condition is necessary and sufficient to design batch-free, fully-online algorithms with vector transports and retractions.
\begin{assumption}\label{assumption3}
    Along the trajectory of the algorithm, we have that the stochastic gradients are unbiased and have bounded fourth central moment, i.e., for each $k \in \{1,\ldots, N\}$, we have $\E_\xi [\grad F(x^k;\xi_k)|\mathcal{F}_{k-1}] = \grad f(x^k)$ and $\E_\xi[\|\grad F(x^k;\xi_k) - \grad f(x^k)\|_{x^{k}}^4|\mathcal{F}_{k-1}]\leq \sigma^4$.
\end{assumption}
Note that Assumption~\ref{assumption3} implies Assumption~\ref{assumption0_1}. To proceed with the convergence analysis of Algorithm \ref{algorithm2_vec_tran}, we also need to assume that the retraction we use in Algorithm \ref{algorithm2_vec_tran} is a second-order retraction, as in Assumption \ref{assumption4}. 

\begin{assumption}\label{assumption4}
    The retraction we use in Algorithm \ref{algorithm2_vec_tran} is a second order retraction, i.e. $\forall \xi\in \T_{x}\M$, we have $\mathsf{d}(\retr_{x}(\xi), \Exp_{x}(\xi)) \leq C \|\xi\|_x^2$.
\end{assumption}
Note that the notion of second order retraction is only a local property, i.e., the above inequality only holds when $\|\xi\|$ is not too large. We refer to second order retraction without this locality restriction, since we assume the compactness of $\M$ in Assumption \ref{assumption_compactness} and thus the condition in Assumption~\ref{assumption4} also holds for large $\|\xi\|$ and the constant $C$ will globally depend on the curvature of the manifold. We also point out that the condition in Assumption~\ref{assumption4} is satisfied by projectional retractions; see, for example,~\cite[Proposition 2.2]{absil2012projection}. The study of higher-order (better) approximation to the exponential mapping by the retractions is still an on-going research topic~\cite{gawlik2018high}, while here we only need a second-order retraction.

The following result in Lemma \ref{lemma_comparison}, which is a standard comparison-type result, will be utilized in the subsequent proof.
\begin{lemma}[Theorem 6.5.6 in \cite{burago2022course}]\label{lemma_comparison}
    Suppose the sectional curvature of $\M$ is upper bounded, then $\forall\xi,\eta\in \T_{x}\M$, we have $
        \|\xi - \eta\|_x\leq C\,\mathsf{d}(\Exp_{x}(\xi), \Exp_{x}(\xi))$,     without loss of generality we assume the constant to be $C=1$ for the rest of the paper.
\end{lemma}

The following result shows that with a second-order retraction, the smoothness with respect to exponential mapping implies the smoothness with respect to retractions.
\begin{lemma}\label{lemma_exp_retr_smoothness}
    Suppose Assumption \ref{assumption0_2}, \ref{assumption1} and \ref{assumption_compactness} hold, if the retraction we use in Algorithm \ref{algorithm2_vec_tran} and \eqref{zeroth_order_estimator_retr} satisfy Assumption \ref{assumption4}, then there exists a parameter $L'>0$, such that $f$ is also $L'$-smooth with the retraction, i.e., 
$ |f(\retr_{x}(\eta)) - f(x) - \langle\grad f(x), \eta\rangle_x|\leq \frac{L'}{2}\|\eta\|_x^2,\ \forall \eta\in \T_{x}\M.$  From now on, we denote $L$ as the parameter that satisfies both Assumption~\ref{assumption0_2} and Lemma~\ref{lemma_exp_retr_smoothness} for brevity.
\end{lemma}
\begin{proof}
    Denote $y=\retr_x(\eta)$. Note that we have
$|f(y) - f(x) - \langle\grad f(x), \eta\rangle_x| \leq  |f(y) - f(x) - \langle\grad f(x), \Exp_{x}^{-1}(y)\rangle_x| + |\langle\grad f(x), \Exp_{x}^{-1}(y) - \eta\rangle_x| \leq  L \|\Exp_{x}^{-1}(y)\|_x^2 + \|\grad f(x)\|_x \|\eta - \Exp_{x}^{-1}(y)\|_x \leq  L \|\eta\|_x^2 + G \mathsf{d}(\Exp_{x}(\eta), y)  \leq  (L+G C) \|\eta\|_x^2 =: L' \|\eta\|_x^2$,  where the first inequality is by Assumption \ref{assumption0_1}, the second is by Assumption \ref{assumption1} and Lemma \ref{lemma_comparison}, and the last inequality is by Assumption \ref{assumption4}.
\end{proof}

We remind the readers that Lemma \ref{lemma_exp_retr_smoothness} can guarantee that the retraction-based zeroth-order estimator \eqref{zeroth_order_estimator_retr} still satisfies Lemma \ref{bound_zo_estimator}. In addition, we have the following bound on the fourth moment of $G_{\mu}^{\retr}$.
\begin{lemma}\label{bound_zo_estimator_4th}
    Consider $G_{\mu}$ given by \eqref{zeroth_order_estimator_retr}. Under Assumptions \ref{assumption0_2}, \ref{assumption1}, \ref{assumption_compactness} and \ref{assumption3}, we have $\E\|G_{\mu}^{\retr}(x)\|_{x}^4 \leq \frac{\mu^4 L^4}{2 }(d+12)^6 + 3d^2\|\grad f(x)\|_{x}^4$, where the expectation is taken toward the Gaussian vectors when constructing $G_{\mu}$ and the random variable $\xi$.
\end{lemma}
\begin{proof}
    Since $\E\|G_{\mu}^{\retr}(x)\|_{x}^4 = \frac{1}{\mu^4}\E_{u}[(f(\retr_x(\mu u)) - f(x))^4\|u\|_{x}^4]$ and 
    \begin{align*}
    &(f(\retr_x(\mu u)) - f(x))^4 \\ = & (f(\retr_x(\mu u)) - f(x) - \langle\grad f(x), \mu u\rangle_x + \langle\grad f(x), \mu u\rangle_x)^4 \\ \leq &  8(f(\retr_x(\mu u)) - f(x) - \langle\grad f(x), \mu u\rangle_x)^4+8(\langle\grad f(x), \mu u\rangle_x)^4 \\ \leq &  8 \left(\frac{L}{2}\|\mu u\|_x^2\right)^4 + 8(\langle\grad f(x), \mu u\rangle_x)^4,
    \end{align*}
    where the last inequality is by Lemma \ref{lemma_exp_retr_smoothness}. Therefore we have \begin{align*}
    \E\|G_{\mu}^{\retr}(x)\|_{x}^4 & \leq \frac{\mu^4 L^4}{2}\E\|u\|_{x}^{12} + 8\E[\langle\grad f(x), u\rangle_x^4\|u\|_{x}^4] \\ & 
    \leq \frac{\mu^4 L^4}{2}(d+12)^6 + 8\E[\langle\grad f(x), u\rangle_x^4\|u\|_{x}^4],
    \end{align*}
    where the last inequality is by Lemma 2 in \cite{li2022stochastic}. It remains to bound the last term on the right hand side, and we apply the same trick as in Proposition 1 in \cite{li2022stochastic} here. Since $u$ is an Gaussian vector on the tangent space $\T_x\M$ (dimension is $d$), we can calculate the expectation using the integral directly (denote $g=\grad f(x)$ and omit the subscript $x$ for simplicity):
    \begin{align*}
    \begin{aligned}
    & \E(\| \langle\grad f(x),u\rangle u\|^4) = \frac{1}{\kappa(d)}\int_{\RR^d}\langle g, x\rangle^4 \|x\|^4 e^{-\frac{1}{2}\|x\|^2}d x \\
    \leq & \frac{1}{\kappa(d)}\int_{\RR^d}\|x\|^4 e^{-\frac{\tau}{2}\|x\|^2}\langle g, x\rangle^4 e^{-\frac{1-\tau}{2}\|x\|^2}d x \leq \frac{1}{\kappa(d)}\bigg(\frac{4}{\tau e}\bigg)^2\int_{\RR^d}\langle g, x\rangle^4 e^{-\frac{1-\tau}{2}\|x\|^2}d x \\
    = &\frac{1}{\kappa(d)}\bigg(\frac{4}{\tau e}\bigg)^2\bigg(\frac{1}{1-\tau}\bigg)^{d/2-2}\int_{\RR^d} \langle g, x\rangle^4 e^{-\frac{1}{2}\|x\|^2}d x = 48\bigg(\frac{1}{\tau e}\bigg)^2\bigg(\frac{1}{1-\tau}\bigg)^{d/2-2}\|g\|^4,
    \end{aligned}
    \end{align*}
    where $\kappa(d):=\int_{\RR^d}e^{-\frac{1}{2}\|x\|^2}d x$ is the constant that normalizes Gaussian distribution, the second inequality is by the following fact: $x^p e^{-\frac{\tau}{2}x^2}\leq (\frac{p}{\tau e})^{p/2}$, the second equality is by change of variables and the last equality is by $\E_{x\sim\mathcal{N}(0,I_d)}\langle g, x\rangle^4=3\|g\|^4$. Taking $\tau = {4}/{d}$ gives the desired result.
\end{proof}

We now provide the convergence result for \texttt{Zo-RASA} (Algorithm \ref{algorithm2_vec_tran}). We remind the readers that we assume $C=1$ in both Assumptions \ref{assumption1} and \ref{assumption4}. We would first need to utilize the following Lemma \ref{lemma_zo_gk_gradk_vec_trans}, which is an analog to Lemma \ref{lemma_zo_gk_gradk}.

\begin{lemma}\label{lemma_zo_gk_gradk_vec_trans}
    Suppose Assumptions \ref{assumption0_2}, \ref{assumption1}, \ref{assumption_compactness}, \ref{assumption3} and \ref{assumption4} hold, and $\{x^k,g^k\}$ is generated by Algorithm~\ref{algorithm2}. 
    We have
    \begin{equation*}
    \begin{split}
        \E\|g^{k} - \grad f(x^{k})\|_{x^{k}}^2 
        \leq \Gamma_{k}\tilde{\sigma}_0^2 + \Gamma_{k}\sum_{i=1}^{k}\Big( \frac{(1+\tau_{i-1})\tau_{i-1}}{\Gamma_{i}}\frac{L^2\|g^{i-1}\|_{x^{i-1}}^2}{\beta^2}+\frac{\tau_{i-1}^2}{\Gamma_{i}}\tilde{\sigma}_{i-1}^2+\tau_k\hat{\sigma}^2 \Big),
    \end{split}
    \end{equation*}
    where the expectation $\E$ is taken with respect to all random variables up to iteration $k$, including the Gaussian variables $\{u_i\}_{i=1}^k$ in the zeroth-order estimator~\eqref{zeroth_order_estimator}, and $\tilde{\sigma}_k^2$ is defined in \eqref{def-tilde-sigma}. Further, from the definition of $\tau_k$ in \eqref{tau-2-choices}, we have 
    \begin{equation*}
    \begin{split}
        \sum_{k=1}^{N}\tau_k\E\|g^{k} - \grad f(x^{k})\|_{x^{k}}^2 \leq \sum_{k=0}^{N-1}\bigg( (1+\tau_{k})\tau_{k}\frac{L^2\E\|g^{k}\|_{x^{k}}^2}{\beta^2} + \tau_{k}^2\tilde{\sigma}_k^2+\tau_k\hat{\sigma}^2 \bigg)+\tilde{\sigma}_0^2, \\
        \sum_{k=1}^{N}\tau_k^2\E\|g^{k} - \grad f(x^{k})\|_{x^{k}}^2 \leq \sum_{k=0}^{N-1}\bigg( (1+\tau_{k})\tau^2_{k}\frac{L^2\E\|g^{k}\|_{x^{k}}^2}{\beta^2} + \tau_{k}^3\tilde{\sigma}_k^2+\tau^2_k\hat{\sigma}^2 \bigg)+ \sum_{k=1}^{N}\tau_k^2\tilde{\sigma}_0^2.
    \end{split}
    \end{equation*}
\end{lemma}

\begin{proof}
The proof is almost identical to the proof of Lemma \ref{lemma_zo_gk_gradk}, and we thus omit the details. Note that here we need to utilize Assumption \ref{assumption1} to show $\mathsf{d}(x^i,x^{i+1})^2 \leq t_i^2\|g^i\|_{x^i}^2$.
\end{proof}

To show the bound for the term $\E\|P_{x^{k+1}}^{x^{k}}g^{k+1} - g^k\|_{x^k}^2$, we further need to utilize the following bound for $\|g^k\|_{x^{k}}$ first.

\begin{lemma}\label{lemma_zo_gk_bounded}
    Consider $g^k$ given by Algorithm \ref{algorithm2_vec_tran}. Suppose Assumption \ref{assumption0_2}, \ref{assumption1}, \ref{assumption_compactness}, \ref{assumption3} and \ref{assumption4} hold. Then, we have $
        \E\|g^{k}\|_{x^{k}}^2\leq \mu^2 L^2(d + 6)^3 + 2(d+4) G^2$ and $\E\|g^{k}\|_{x^{k}}^4\leq \frac{\mu^4L^4}{2}(d+12)^6 + 3d^2 G^4$,  where the expectation $\E$ is taken with respect to all random variables up to iteration $k$.
\end{lemma}
\begin{proof}
    Note that we have 
    \begin{align*}
    \|g^{k}\|_{x^{k}}^2 & = \|(1-\tau_{k-1}) \mathcal{T}^{k-1} (g^{k-1}) + \tau_{k-1} \mathcal{T}^{k-1} (G_{\mu}^{k-1})\|_{x^{k}}^2 \\ & \leq (1-\tau_{k-1})\|g^{k-1}\|_{x^{k-1}}^2 + \tau_{k-1}\|G_{\mu}^{k-1}\|_{x^{k-1}}^2.
    \end{align*}
    Taking expectation conditioned on $\mathcal{F}_{k-1}$, we have by Lemma \ref{bound_zo_estimator} that $
        \E[\|g^{k}\|_{x^{k}}^2|\mathcal{F}_{k-1}]\leq (1-\tau_{k-1})\E\|g^{k-1}\|_{x^{k-1}}^2 + \tau_{k-1}(\mu^2 L^2(d + 6)^3 + 2(d+4) \|\grad f(x^{k-1})\|_{x^{k-1}}^2)$. We remove the conditional expectation by law of total expectation, also by Assumption \ref{assumption_compactness} we have that 
    \begin{align*}
        \E\|g^{k}\|_{x^{k}}^2\leq (1-\tau_{k-1})\E\|g^{k-1}\|_{x^{k-1}}^2 + \tau_{k-1}(\mu^2 L^2(d + 6)^3 + 2(d+4) G^2).
    \end{align*}
    Denote $A_k=\E\|g^{k}\|_{x^{k}}^2$, note that we have $A_k\leq (1-\tau_{k-1})A_{k-1} + \tau_{k-1}(\mu^2 L^2(d + 6)^3 + 2(d+4) G^2)$. Again from Lemma \ref{bound_zo_estimator} we have $A_0 \leq \mu^2 L^2(d + 6)^3 + 2(d+4) G^2$, from which and using induction, we conclude that $ A_k = \E\|g^{k}\|_{x^{k}}^2\leq \mu^2 L^2(d + 6)^3 + 2(d+4) G^2$. As for the fourth moment, note that
    \begin{align*}
        \E&(\|g^{k}\|_{x^{k}}^2)^2 \leq  \E\left((1-\tau_{k-1})\|g^{k-1}\|_{x^{k-1}}^2 + \tau_{k-1}\|G_{\mu}^{k-1}\|_{x^{k-1}}^2\right)^2 \\
        \leq& (1-\tau_{k-1})\E\|g^{k-1}\|_{x^{k-1}}^4 + \tau_{k-1}\E\|G_{\mu}^{k-1}\|_{x^{k-1}}^4,\\
        \leq & (1-\tau_{k-1})\E\|g^{k-1}\|_{x^{k-1}}^4 + \tau_{k-1}\bigg(\frac{\mu^4 L^4}{2 }(d+12)^6 + 3d^2\|\grad f(x^k)\|_{x^k}^4\bigg)
    \end{align*}
    where the last inequality is by Lemma \ref{bound_zo_estimator_4th}. The final result follows similarly to the second moment case.
\end{proof}

Now we are ready to study the difference between $g^k$ and $g^{k+1}$.
\begin{lemma}\label{lemma_zo_sum_gkplus1_gk_retr}
    Suppose Assumptions \ref{assumption0_2}, \ref{assumption1}, \ref{assumption_compactness}, \ref{assumption3} and \ref{assumption4} hold, and take $\tau_k$ as in \eqref{tau-2-choices}. Then, we have
    \begin{align}\label{zo_sum_gkplus1_gk_retr}
    \begin{split}
        \sum_{k=1}^{N}\E\|P_{x^{k+1}}^{x^{k}}g^{k+1} &- g^k\|_{x^k}^2\leq  \frac{4L^2}{\beta^2}\sum_{k=0}^{N-1}(1+\tau_{k})\tau^2_{k}\E\|g^{k}\|_{x^{k}}^2 + 4\sum_{k=0}^{N}(\tau_k^2+\tau_k^3)\Tilde{\sigma}_k^2\\& +\left[4\tilde{\sigma}_0^2 + 4\hat{\sigma}^2 + \frac{8}{\beta^2}\bigg(\frac{\mu^4L^4}{2}(d+12)^6 + 3d^2 G^4\bigg)\right]\sum_{k=0}^{N}\tau_k^2,
    \end{split}
    \end{align}
    where the expectation $\E$ is taken with respect to all random variables up to iteration $k$, which includes the random variables $u$ in the zeroth-order estimator \eqref{zeroth_order_estimator_retr}.
\end{lemma}
\begin{proof}
    Since
    \begin{align*}
    \begin{aligned}
        &\|P_{x^{k+1}}^{x^{k}}g^{k+1} - g^k\|_{x^{k}}^2 = \|g^{k+1} - P_{x^{k}}^{x^{k+1}}g^k\|_{x^{k+1}}^2 \\
        \leq & 2\|g^{k+1} - \mathcal{T}^k g^k\|_{x^{k+1}}^2 + 2\|\mathcal{T}^k g^k - P_{x^{k}}^{x^{k+1}}g^k\|_{x^{k+1}}^2 \\
        \leq & 2\tau_k^2 \|G_{\mu}^k - g^k\|_{x^{k}}^2+2\mathsf{d}(x^{k+1}, x^{k})^2\|g^k\|_{x^{k}}^2\\
        \leq & 4\tau_k^2 \|G_{\mu}^k - \grad f(x^k)\|_{x^{k}}^2 + 4\tau_k^2 \|\grad f(x^k) - g^k\|_{x^{k}}^2 +2\frac{\tau_k^2}{\beta^2} \|g^k\|_{x^{k}}^4,
    \end{aligned}
    \end{align*}
    where the second inequality is by the update and Assumption \ref{assumption1}, and the last inequality is by Assumption \ref{assumption1}.
    Now taking the expectation conditioned on $\mathcal{F}_{k-1}$ we get:
    \begin{align*}
        \E[\|P_{x^{k+1}}^{x^{k}}g^{k+1} &- g^k\|_{x^{k}}^2|\mathcal{F}_{k-1}]\leq 4\tau_k^2 \E[\|G_{\mu}^k - \grad f(x^k)\|_{x^{k}}^2|\mathcal{F}_{k-1}] \\&+ 4\tau_k^2 \E[\|\grad f(x^k) - g^k\|_{x^{k}}^2|\mathcal{F}_{k-1}] +2\frac{\tau_k^2}{\beta^2} \E[\|g^k\|_{x^{k}}^4|\mathcal{F}_{k-1}].
    \end{align*}
    Thus we have (by law of total expectation):
    \begin{equation*}
    \begin{split}
        & \sum_{k=1}^{N}\E\|P_{x^{k+1}}^{x^{k}}g^{k+1} - g^k\|_{x^{k}}^2\\
        \leq & 4\sum_{k=1}^{N}\tau_k^2\E\|G_{\mu}^k - \grad f(x^k)\|_{x^{k}}^2 + 4\sum_{k=1}^{N}\tau_k^2\E\|\grad f(x^k) - g^k\|_{x^{k}}^2+\frac{2}{\beta^2} \sum_{k=1}^{N}\tau_k^2\E\|g^k\|_{x^{k}}^4 \\
        \leq& 4\sum_{k=1}^{N}\tau_k^2\Tilde{\sigma}_k^2 + 4\sum_{k=1}^{N}\tau_k^2\E\|\grad f(x^k) - g^k\|_{x^{k}}^2+\frac{8}{\beta^2}\bigg(\frac{\mu^4L^4}{2}(d+12)^6 + 3d^2 G^4\bigg) \sum_{k=1}^{N}\tau_k^2 
    \end{split}
    \end{equation*}
    where the second inequality is by Lemmas \ref{bound_zo_estimator} and \ref{lemma_zo_gk_bounded}. The desired result follows by applying Lemma \ref{lemma_zo_gk_gradk_vec_trans} to the above inequality.
\end{proof}

We now state the main result in Theorem \ref{theorem2_vec}, as an analog to Theorem \ref{theorem3}. Notice that different from Theorem \ref{theorem3}, we do not need $N=\Omega(d)$ in case (ii), in view of Remark \ref{rmk_bdd_grad_no_d} and Assumption \ref{assumption_compactness}.

\begin{theorem}\label{theorem2_vec}
    Suppose Assumptions \ref{assumption0_2}, \ref{assumption1}, \ref{assumption_compactness}, \ref{assumption3} and \ref{assumption4} hold. In Algorithm \ref{algorithm2_vec_tran}, we set $
     \mu=\mathcal{O}\big(\frac{1}{L d^{3/2}N^{1/4}}\big)$ and $
       \beta\geq \sqrt{d} L$. Then the following holds.
    \begin{itemize}
        \item[(i)] If we choose $\tau_0=1$, $\tau_k= {1}/{\sqrt{N}}$, $k\geq 1$ and $m_k\equiv 8(d+4)$, $k\geq 0$, then we have $\frac{1}{N+1}\sum_{k=0}^{N} \E\|\grad f(x^{k})\|_{x^{k}}^2 \leq \mathcal{O}({1}/{\sqrt{N}})$.
        \item[(ii)] If we choose $\tau_0=1$, $\tau_k= {1}/{\sqrt{dN}}$, $k\geq 1$, $m_0=d$ and $m_k=1$ for $k\geq 1$, then we have $\frac{1}{N+1}\sum_{k=0}^{N} \E\|\grad f(x^{k})\|_{x^{k}}^2 \leq \mathcal{O}(\sqrt{{d}/{N}})$. 
    \end{itemize}
    Here the expectation $\E$ is taken with respect to all random variables up to iteration $k$, which includes the random variables $u$ in zeroth-order estimator \eqref{zeroth_order_estimator_retr}.
\end{theorem}
\begin{proof}[Proof of Theorem \ref{theorem2_vec}]
    The proof is very similar to the proof of Theorem \ref{theorem3}. We first will have the following inequality analogue to \eqref{zo_temp5}:
    \begin{align*}
    \begin{aligned}
        \frac{1}{8\beta^2}\sum_{k=0}^{N}\tau_k\E\|g^k\|_{x^k}^2 \leq & W^0 + \frac{1}{2\beta}\sum_{k=0}^{N}\tau_k\hat{\sigma}^2 +\frac{2}{\beta}\sum_{k=0}^{N}(\tau_k^2+\tau_k^3)\Tilde{\sigma}_k^2\\&+\frac{1}{2\beta}[4\Tilde{\sigma}_0^2+4\hat{\sigma}^2+\frac{8}{\beta^2}(\frac{\mu^2L^2}{2}(d+12)^6+3d^2G^4)]\sum_{k=0}^{N}\tau_k^2
    \end{aligned}
    \end{align*}
    Note that we still need \eqref{zo_temp_ineq1} to show the above inequality.
    
    We then directly provide the result corresponding to \eqref{zo_temp7}:
    \begin{align}\label{zo_temp7_vec}
    \begin{aligned}
        \sum_{k=1}^{N}\frac{\tau_k}{2}& \E\|\grad f(x^{k})\|_{x^{k}}^2
        \leq (8\beta^2+16 L^2)\bigg(W^0 + \frac{1}{2\beta}\sum_{k=0}^{N}\tau_k\hat{\sigma}^2 +\frac{2}{\beta}\sum_{k=0}^{N}(\tau_k^2+\tau_k^3)\Tilde{\sigma}_k^2\\&+\frac{1}{2\beta}[4\Tilde{\sigma}_0^2+4\hat{\sigma}^2+\frac{8}{\beta^2}(\frac{\mu^2L^2}{2}(d+12)^6+3d^2G^4)]\sum_{k=0}^{N}\tau_k^2\bigg) + \sum_{k=0}^{N-1}\tau_k^2\Tilde{\sigma}_k^2+\sum_{k=0}^{N-1}\tau_k^2\hat{\sigma}^2+\Tilde{\sigma}_0^2
    \end{aligned}
    \end{align}
    Now by Assumption \ref{assumption_compactness}, we have $\Tilde{\sigma}_k^2\leq \sigma_k^2+\frac{8(d+4)}{m_k}G^2$, which is exactly the reason we don't need to show an inequality similar to \eqref{zo_temp_ineq2}.

    For case (i) in Theorem~\ref{theorem2_vec}, \eqref{zo_temp7_vec} can be rewritten as 
    \[
    \frac{1}{N+1}\sum_{k=0}^{N} \E\|\grad f(x^{k})\|_{x^{k}}^2 
    \leq \frac{c_1W(x^0, g^0)}{\sqrt{N}}+ c_2\hat{\sigma}^2 + \frac{c_3\frac{1}{N}\sum_{k=0}^{N}\tilde{\sigma}_{k}^2}{\sqrt{N}} + \frac{c_4}{\sqrt{N}}\Tilde{\sigma}_0^2,
    \]
    for some absolute positive constants $c_1$, $c_2$, $c_3$ and $c_4$. The proof for case (i) is completed by noting that (see \eqref{def-tilde-sigma}) $\hat{\sigma}^2=\mathcal{O}(1/\sqrt{N})$, $\frac{1}{N}\sum_{k=0}^{N}\Tilde{\sigma}_{k}^2=\mathcal{O}(1)$ and $\Tilde{\sigma}_0^2=\mathcal{O}(1)$. 
    
    For case (ii) in Theorem \ref{theorem2_vec}, \eqref{zo_temp7_vec} can be rewritten as 
    \[
    \frac{1}{N+1}\sum_{k=0}^{N} \E\|\grad f(x^{k})\|_{x^{k}}^2 
    \leq c_1' W(x^0, g^0)\sqrt{\frac{d}{N}}+ c_2'\hat{\sigma}^2 + \frac{c_3'\frac{1}{N}\sum_{k=0}^{N}\tilde{\sigma}_{k}^2}{\sqrt{d N}} + c_4'\sqrt{\frac{d}{N}}\Tilde{\sigma}_0^2,
    \]
    for some positive constants $c_1'$, $c_2'$, $c_3'$ and $c_4'$. The proof of case (ii) is completed by noting that $\Tilde{\sigma}_0^2=\mathcal{O}(1)$, $\hat{\sigma}^2=\mathcal{O}(1/\sqrt{N})$ and $\frac{1}{N}\sum_{k=0}^{N}\Tilde{\sigma}_{k}^2=\mathcal{O}(d)$.
\end{proof}

\begin{remark}
By the technique discussed in Remark~\ref{sec:samplingtrick}, to obtain an $\epsilon$-approximate stationary point in Definition~\ref{def:epsstat} we need an oracle complexity of $\mathcal{O}(d/\epsilon^4)$.
\end{remark}

\section{Numerical experiments}\label{sec:experiments}

\subsection{$k$-PCA}

We now provide numerical results on the $k$-PCA problem to demonstrate the effectiveness of the Zo-RASA algorithms. For a given centered random vector $\mathbf{z}\in\mathbb{R}^n$, the $k$-PCA problem corresponds to finding the subspace spanned by the top-$k$ eigenvectors of its positive definite covariance matrix $\Sigma=\E[\mathbf{z} \mathbf{z}^\top]$. Formally, we have the following problem on the Stiefel manifold: 
\begin{align}\label{problem_kPCA}
\min_{X\in\St(n, r)} f(X) := -\frac{1}{2}\tr(X^\top \E[\mathbf{z} \mathbf{z}^\top] X).
\end{align}
Note that the dimension of the Stiefel is given by $d=nr-r(r+1)/2$.

For any $Y=XQ$ where $Q\in\RR^{r\times r}$, and $Q^\top Q = QQ^\top=I_r$, we have $f(X)=f(Y)$. Hence, we can equivalently view~\eqref{problem_kPCA} as the following minimization problem on the Grassmann manifold:
\begin{align*}
\min_{[X]\in\Gr(n, r)} f([X]) := -\frac{1}{2}\tr(X^\top \E[\mathbf{z} \mathbf{z}^\top] X).
\end{align*}
Note that the dimension of the Grassmannian is given by $d=r(n-r)$.

We solve \eqref{problem_kPCA} using Algorithm \ref{algorithm2_vec_tran} and compare it with the zeroth-order Riemannian SGD method from~\cite{li2022stochastic}. In all the experiments, we used projecting vector transport rather than parallel transport for Stiefel manifolds, due to the aforementioned facts that parallel transport is time-consuming to numerically compute on Stiefel manifold, and has no closed form. In the stochastic zeroth-order setting, for each query point $X_k$, the stochastic oracle returns a noise estimate of $f(x)$ based on a single observation $\mathbf{z}_k$, i.e. $F(X^{k};\mathbf{z}_k)=-1/2\tr((X^k)^\top \mathbf{z}_k \mathbf{z}_k^\top X^k)$. For our experiments, we assume $\mathbf{z}_k$ is sampled from a centered Gaussian distribution with covariance matrix given by $
\Sigma = \sum_{i=1}^r \lambda_i v_i v_i^\top + \sum_{i=r+1}^{n} \lambda_i v_i v_i^\top$,
where $V=[v_1, ..., v_n]$ is an orthogonal matrix. The first $r$ $\lambda_i$s are uniform random numbers in $[100, 200]$ and the last $n-r$ are uniform random numbers in $[1, 50]$. For our experiments, we fix $r$ and try different $n$ (reflected in different rows in Figure \ref{fig:kpca2}). 

We set $N=50000\times n$ for \texttt{Zo-RASA} and one-batch Zo-RSGD (\texttt{Zo-RSGD-1}) algorithms, while $N=50000$ for our mini-batch Zo-RSGD algorithm (\texttt{Zo-RSGD-m}). The reason here is that for \texttt{Zo-RSGD-m}, we take $m=n=\mathcal{O}(d)$ since we fix $r$ and change $n$. While the theoretical result in~\cite{li2022stochastic} requires the batch-size $m$ to be $\mathcal{O}(d/\epsilon^2)$, they empirically observed reasonable-order batch-sizes suffices. For \texttt{Zo-RASA}, according to our theory, we again take $\tau_k=0.01/\sqrt{N}$ and $\beta=100$. For \texttt{Zo-RSGD-1} and \texttt{Zo-RSGD-m}, we set $t_k$ as $t_k=10^{-4}/\sqrt{N}$ and  $t_k=5\times10^{-4}/\sqrt{N}$ respectively.

For all algorithms, we again compare the function value, norm of the Riemannian gradient and the principal angles between the current iterate and the optimal subspace. Figures~\ref{fig:kpca2} plots the results. The experimental results provide support for the proposed algorithms (and the established theory), demonstrating that the proposed Zo-RASA algorithm is more efficient in terms of decreasing the Riemannian gradient and principal angles compared to conventional zeroth-order Riemannian stochastic gradient descent methods that utilize mini-batches.

\begin{figure}[t!]
    \begin{center}
    \subfigure{\includegraphics[width=0.32\textwidth]{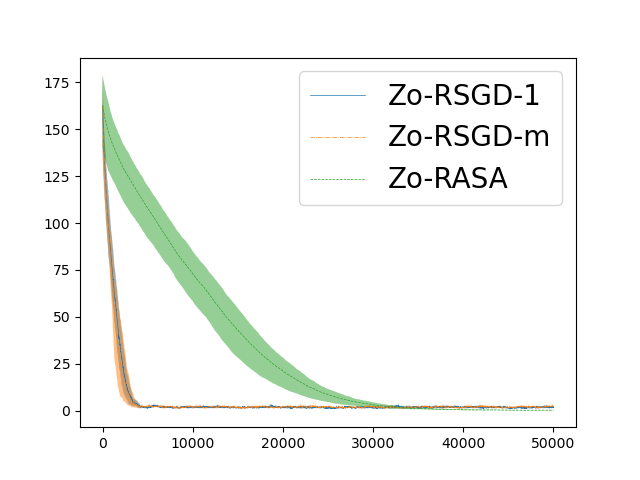}}
    \subfigure{\includegraphics[width=0.32\textwidth]{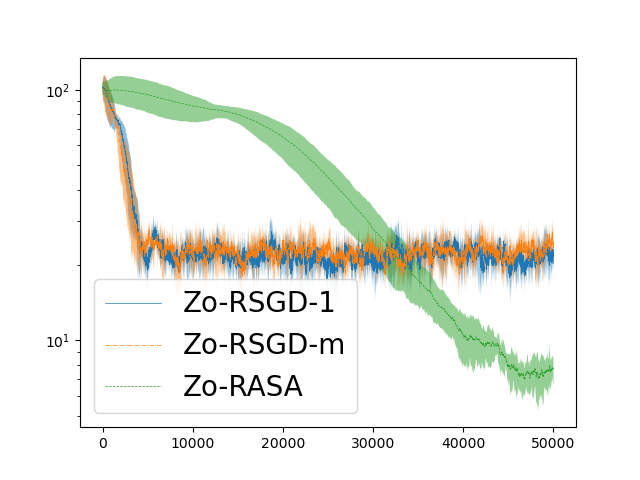}}
    \subfigure{\includegraphics[width=0.32\textwidth]{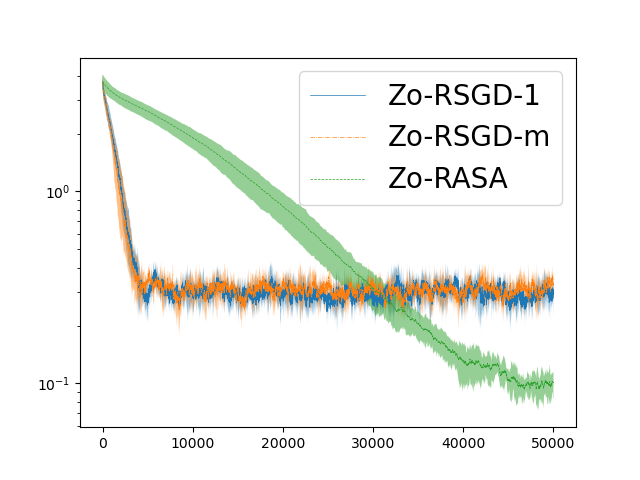}}
    
    \subfigure{\includegraphics[width=0.32\textwidth]{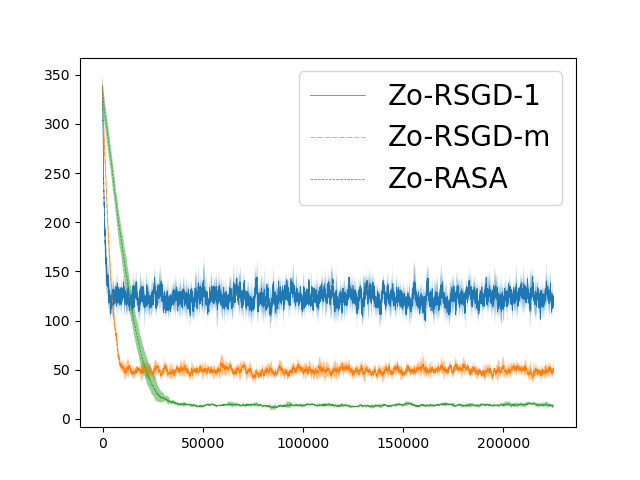}}
    \subfigure{\includegraphics[width=0.32\textwidth]{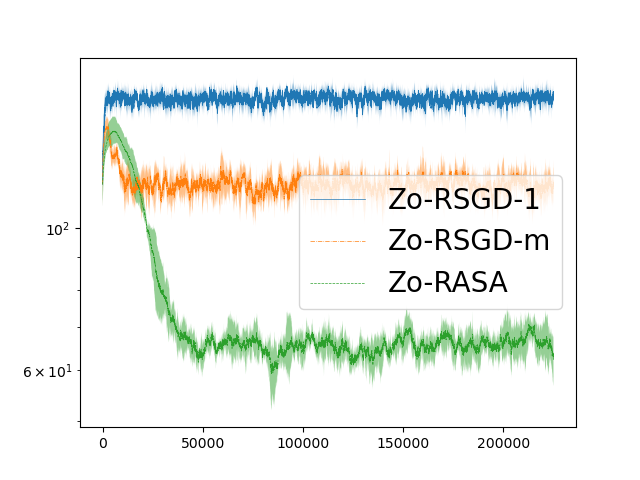}}
    \subfigure{\includegraphics[width=0.32\textwidth]{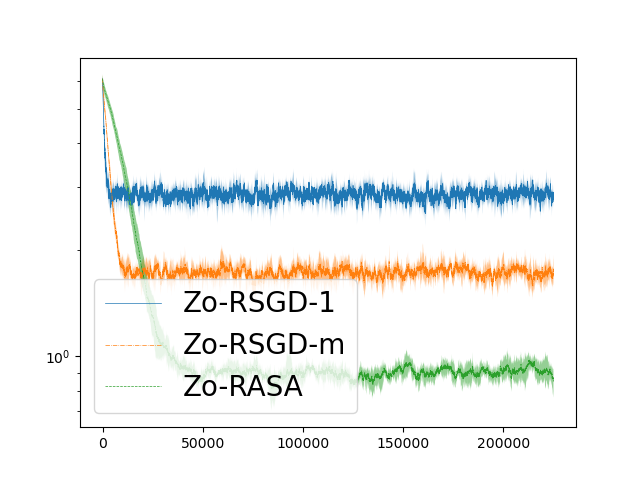}}
    
    \setcounter{subfigure}{0}
    \subfigure[Optimality gap]{\includegraphics[width=0.32\textwidth]{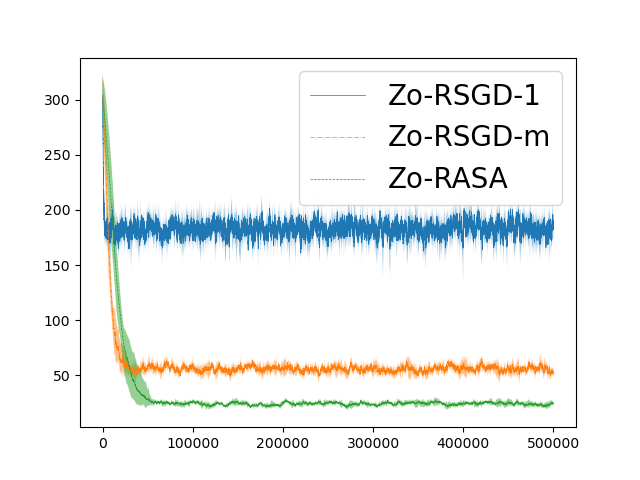}}
    \subfigure[$\|\grad f(X^t)\|$]{\includegraphics[width=0.32\textwidth]{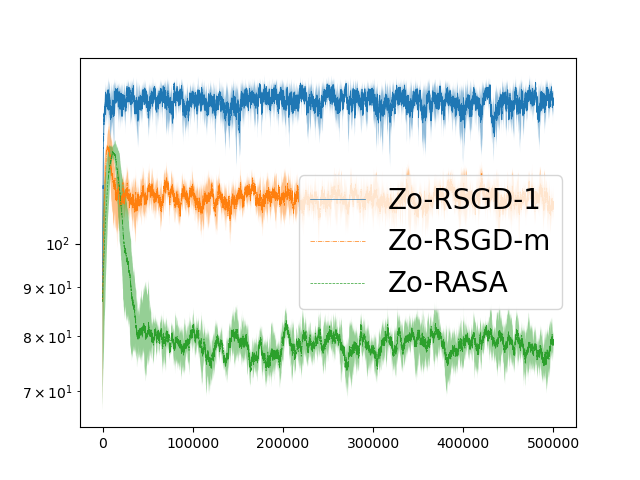}}
    \subfigure[Principal angles]{\includegraphics[width=0.32\textwidth]{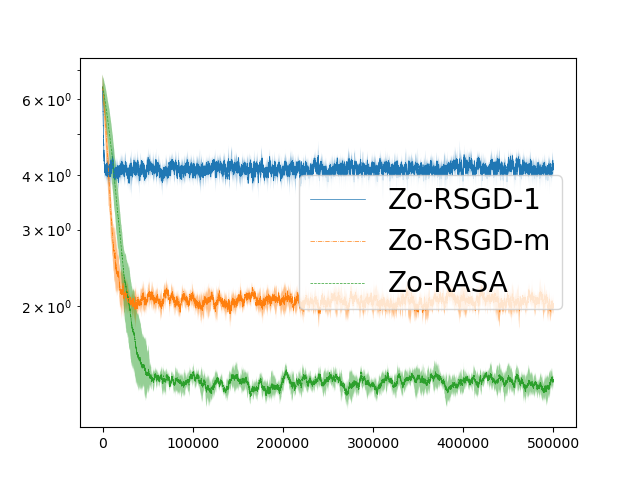}}
    
    \caption{Results for kPCA \eqref{problem_kPCA} with $n\in\{10, 30, 50\}$ (corresponding to three rows) and $r=5$. The resulting manifold (Stiefel) dimensions are $d=\{35, 135, 235\}$. The x-axis is the number of zeroth-order oracle calls (i.e. number of function value calls).}
    \label{fig:kpca2}
    \end{center}
\end{figure}

\subsection{Identification of a fixed rank symmetric positive semi-definite matrix}
We now provide another numerical example from~\cite{bonnabel2013stochastic}. Consider a matrix-version linear model as in~\cite{tsuda2005matrix}:
\begin{align*}
    y_t=\tr(W \mathbf{x}_t \mathbf{x}_t^\top) = \mathbf{x}_t^\top W \mathbf{x}_t
\end{align*}
where $\mathbf{x}_t\in\RR^n$ is the input and $y_t\in\RR$ is the output, and the unknown matrix $W\in\RR^{n\times n}$ is a positive semi-definite matrix with a fixed rank $r$ ($r\leq n$). Denote the set
\begin{align}\label{eq_frpd_manifold}
    S_{+}(n, r) = \{W\in\RR^{n\times n}| W=W^\top,\mathrm{rank}(W)=r\}
\end{align}
which is the set of positive definite matrices with rank $r$. The problem is thus formulated as a matrix least square problem
\begin{align}\label{problem_frpd}
    \min_{W\in S_{+}(n, r)}f(W):=\frac{1}{2}\E_{\mathbf{x},y}(\mathbf{x}^\top W \mathbf{x} - y)^2
\end{align}

Notice that $W$ can be represented as $W=GG^\top$ where $G\in\RR^{n,r}$ is a matrix with full column rank. Also notice that for any orthogonal matrix $O\in\RR^{r\times r}$ we have $W= G O O^\top G^\top=GG^\top$, we have the following quotient representation of the set of fixed rank positive definite matrices $S_{+}(n, r) \simeq \RR_*^{n \times r} / \mathcal{O}(r)$, where the right hand side represents the set of equivalent classes:
\begin{align*}
    [G]=\{GO| O\in\mathcal{O}(r)\}.
\end{align*}
We could thus conduct our experiment on the quotient manifold $\RR_*^{n \times r} / \mathcal{O}(r)$, with the following re-formulated problem:
\begin{align}\label{problem_frpd_re}
    \min_{[G]\in \RR_*^{n \times r} / \mathcal{O}(r)}f(G):=\frac{1}{2}\E_{\mathbf{x},y}(\mathbf{x}^\top G G^\top \mathbf{x} - y)^2
\end{align}

The manifold $S_{+}(n, r)$ has dimension $d = nr - r(r-1)/2$ and is not a compact manifold. We test \eqref{problem_frpd_re} to show the efficiency of our proposed algorithm even without the compactness assumption (Assumption \ref{assumption_compactness}) which we need to conduct our theoretical analysis.

We solve \eqref{problem_frpd_re} using Algorithm \ref{algorithm2_vec_tran} and compare it with the zeroth-order Riemannian SGD method from~\cite{li2022stochastic}. In all the experiments, we used again retraction and projecting vector transport rather than exponential mapping and parallel transport. The ground-truth $G^\star\in\RR^{n\times r}$ is sampled randomly with standard Gaussian entries. For our experiments, we sample $\mathbf{x}\sim \mathcal{N}(0, I_d)$ and construct $y=\mathbf{x}^\top W \mathbf{x}$ noiselessly. Specifically, given a query point $G^t$ and a Gaussian sample $\mathbf{x}_t$ with $y_t=\mathbf{x}_t^\top G^\star (G^\star)^\top \mathbf{x}_t$, the stochastic zeroth-order oracle gives the value $\frac{1}{2}(\mathbf{x}_t^\top G^t (G^t)^\top \mathbf{x}_t - y_t)^2$. For our experiments, we fix $r$ and test with different $n$  (reflected in different rows in Figure \ref{fig:frpd}). 

We set $N=5000\times n$ for \texttt{Zo-RASA} and one-batch Zo-RSGD (\texttt{Zo-RSGD-1}) algorithms, while $N=5000$ for our mini-batch Zo-RSGD algorithm (\texttt{Zo-RSGD-m}) for the same reason as the kPCA experiments. For \texttt{Zo-RASA}, according again to our theory, we again take $\tau_k=10^{-3}/\sqrt{N}$ and $\beta=100$. For \texttt{Zo-RSGD-1} and \texttt{Zo-RSGD-m}, we set $t_k=10^{-5}/\sqrt{N}$.

For all algorithms, we again compare the function value, norm of the Riemannian gradient and the quantity $\|G^t (G^t)^\top - G^\star (G^\star)^\top\|$ which measures the error to the ground truth positive semi-definite matrix. Figures~\ref{fig:frpd} plots the results. It's worth noticing here that mini-batch Zo-RSGD seems to work the worst in the plots, which is due to the fact that we take the step sizes the same for \texttt{Zo-RSGD-1} and \texttt{Zo-RSGD-m}. The reason we cannot enlarge the step size for \texttt{Zo-RSGD-m} is that the projectional retraction and projectional vector transport requires solving a Sylvester equation which leads to numerical stability issues if the step sizes become large (see \cite{manopt} for details). The experimental results provide support for the proposed algorithms (and the established theory), demonstrating that the proposed Zo-RASA algorithm is more efficient in terms of decreasing the Riemannian gradient and function values compared to conventional zeroth-order Riemannian stochastic gradient descent methods that utilize mini-batches.

\begin{figure}[t!]
    \begin{center}
    \subfigure{\includegraphics[width=0.32\textwidth]{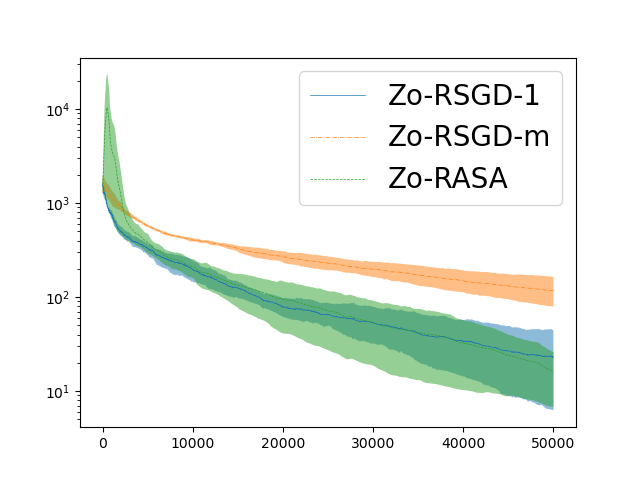}}
    \subfigure{\includegraphics[width=0.32\textwidth]{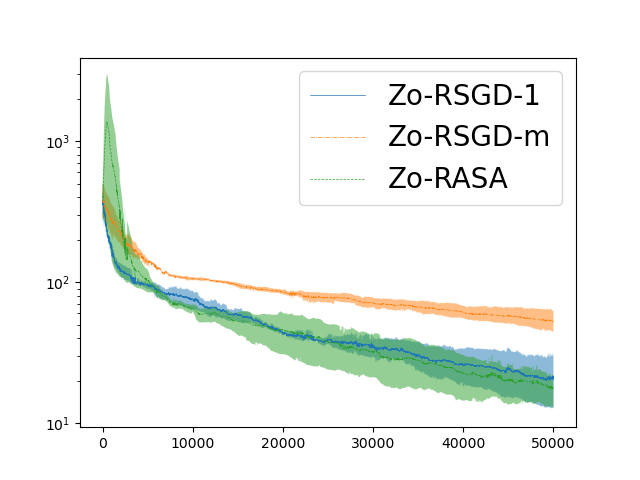}}
    \subfigure{\includegraphics[width=0.32\textwidth]{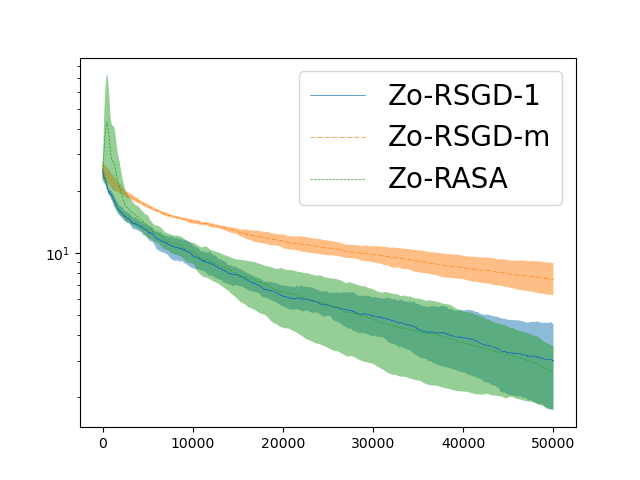}}
    
    \subfigure{\includegraphics[width=0.32\textwidth]{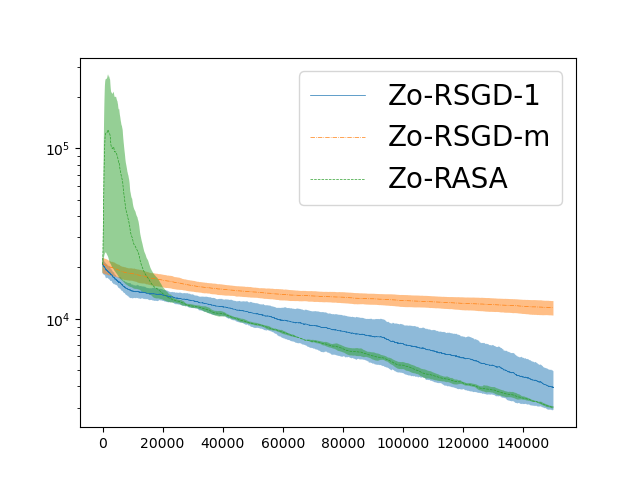}}
    \subfigure{\includegraphics[width=0.32\textwidth]{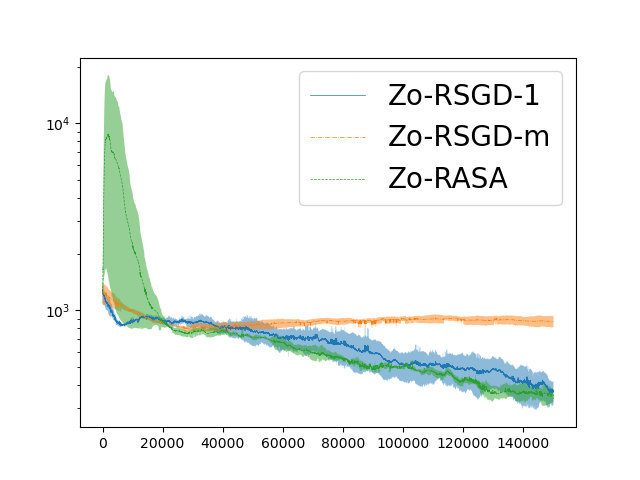}}
    \subfigure{\includegraphics[width=0.32\textwidth]{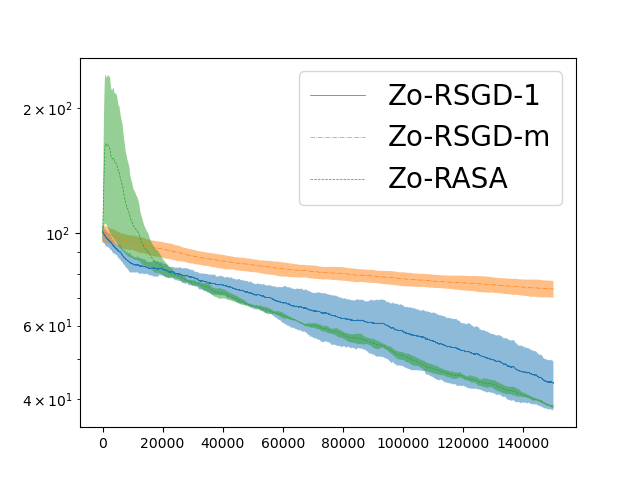}}
    
    \setcounter{subfigure}{0}
    \subfigure[Optimality gap]{\includegraphics[width=0.32\textwidth]{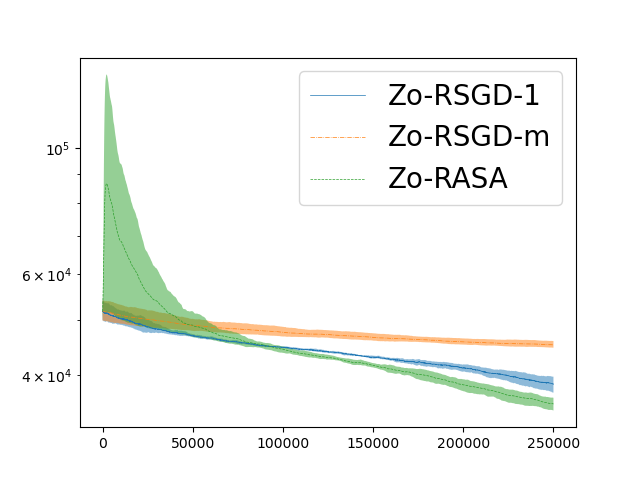}}
    \subfigure[$\|\grad f(G^t)\|$]{\includegraphics[width=0.32\textwidth]{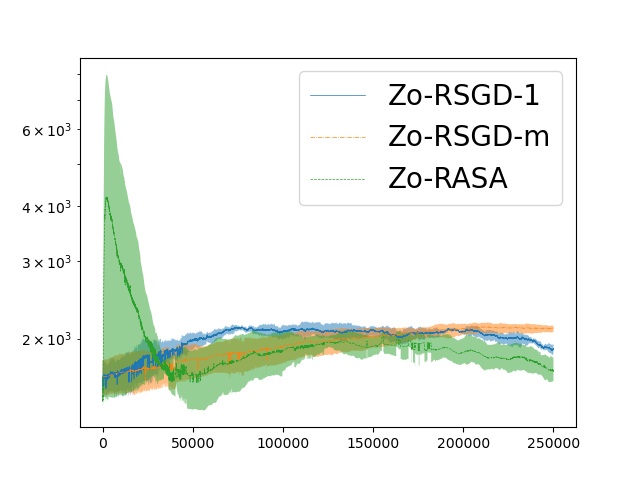}}
    \subfigure[$\|G^t (G^t)^\top - G^\star (G^\star)^\top\|$]{\includegraphics[width=0.32\textwidth]{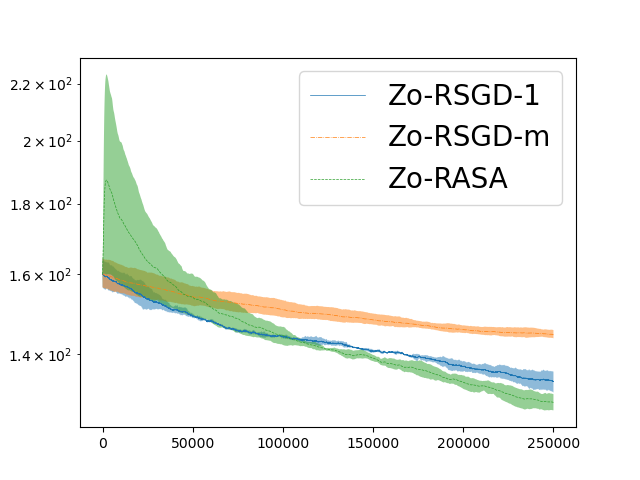}}
    
    \caption{Results for \eqref{problem_frpd_re} with $n\in\{10, 30, 50\}$ (corresponding to three rows) and $r=5$. The resulting manifold as defined in \eqref{eq_frpd_manifold} are $d=\{40, 140, 240\}$ dimensional, respectively. The x-axis is the number of zeroth-order oracle calls (i.e. number of function value calls).}
    \label{fig:frpd}
    \end{center}
\end{figure}

\section*{Acknowledgements} We thank Prof. Otis Chodosh (Stanford) for several helpful discussions and clarifications regarding several differential geometric concepts. JL thanks Xuxing Chen for helpful discussions. KB was supported in part by National Science Foundation (NSF) grant DMS-2053918. SM was supported in part by NSF grants DMS-2243650, CCF-2308597, CCF-2311275 and ECCS-2326591, UC Davis CeDAR (Center for Data Science and Artificial Intelligence Research) Innovative Data Science Seed Funding Program, and a startup fund from Rice University.

\bibliographystyle{abbrvnat} 
\bibliography{reference}
\end{document}